\documentclass[final]{siamart190516}

\usepackage{amsmath, amsfonts, amssymb} 
\usepackage{mathrsfs}
\usepackage{dsfont}
\usepackage{a4}
\usepackage{url}
\usepackage{enumerate}
\usepackage[english]{babel}
\usepackage{graphicx, caption}
\usepackage[authoryear]{natbib}

\def\be{\begin{eqnarray}}
\def\ee{\end{eqnarray}}
\def\b*{\begin{eqnarray*}}
	\def\e*{\end{eqnarray*}}

\newcommand{\rmi}{{\rm (i)$\>\>$}}
\newcommand{\rmii}{{\rm (ii)$\>\>$}}

\newcommand{\E}{\mathbb{E}}
\newcommand{\Lsp}{\mathbf{L}}
\newcommand{\R}{\mathbb{R}}
\newcommand{\Z}{\mathbb{Z}}

\def\Cc{{\mathcal C}}
\def\Dc{{\mathcal D}}
\def\Ec{{\mathcal E}}

\def\Lc{{\mathcal L}}
\newcommand{\Mc}{\mathcal{M}}
\def\Nc{{\mathcal N}}

\def\Pc{{\mathcal P}}
\def\Qc{{\mathcal Q}}

\def\Uc{{\mathcal U}}

\def\Wc{{\mathcal W}}
\def\Xc{{\mathcal X}}

\def \eps {\varepsilon}

\def \la {\lambda}
\def \m {\mu}

\def \vp {\varphi}

\def\q{\quad}

\def \sb {\subset}

\def \no{\noindent}
\def \supp{{\rm{supp}}}

\def \mv{\check{\mu}}

\def \po {\overline{\mathsf{P}}}
\def \mmoto {\overline{\mathsf{MMOT}}}
\def \mmotu {\underline{\mathsf{MMOT}}}
\def \ddo {\overline{\mathsf{D}}}
\def \pu {\underline{\mathsf{P}}}
\def \du {\underline{\mathsf{D}}}
\def \Do {\overline{\Dc}}
\def \Du {\underline{\Dc}}
\def \oo {\overline{\mathsf{OT}}}
\def \ou {\underline{\mathsf{OT}}}

\def\bk{\boldsymbol{K}}
\def\bc{\boldsymbol{C}}

\newcommand{\leqcvx}{\preceq_{cx}}

\usepackage{caption}
\usepackage{graphicx,booktabs}
\usepackage{array}
\newcolumntype{P}[1]{>{\centering\arraybackslash}p{#1}}

\newsiamthm{ass}{Assumption}
\newsiamthm{cond}{Condition}
\newsiamremark{remark}{Remark}
\newsiamremark{example}{Example}

\title{Robust pricing and hedging of options on multiple assets and its 
numerics\thanks{Submitted to the editors 09 September 2019. Accepted 06 October 
2020.\funding{
Support from the European 
Research Council under the European Union's Seventh Framework Programme (FP7/2007-2013) / ERC grant agreement no. 335421 is 
gratefully acknowledged. JO is also thankful to St.\ John's College 
in Oxford for its financial support. TL is grateful for the support of 
ShanghaiTech University, and in addition, to the University of Toronto and its 
Fields Institute for the Mathematical Sciences, where parts of this work were 
performed. GG gratefully acknowledges the support of University of Michigan and 
an AMS Simons Travel Grant from the American Mathematical Society and Simons 
Foundation.}}}
\author{Stephan Eckstein\thanks{Department of Mathematics, University of 
Konstanz (\email{stephan.eckstein@uni-konstanz.de})}
\and Gaoyue Guo\thanks{Laboratoire MICS, CentraleSup\'elec 
(\email{gaoyue.guo@centralesupelec.fr})}
\and Tongseok Lim\thanks{
Krannert School of Management, Purdue University (\email{lim336@purdue.edu})}
\and Jan Ob{\L}{\'o}j\thanks{Mathematical Institute and St.\ John's College, 
University of Oxford (\email{jan.obloj@maths.ox.ac.uk})}}

\begin{document}
\maketitle
\begin{center} forthcoming in SIAM Journal on Financial Mathematics 
\end{center}\vspace{1em}

\begin{abstract}
	We consider robust pricing and hedging for options written on multiple 
	assets given market option prices for the individual assets. The resulting 
	problem is called the multi-marginal martingale optimal transport problem. 
	We propose two numerical methods to solve such problems: using 
	discretisation and linear programming applied to the primal side and using 
	penalisation and deep neural networks optimisation applied to the dual 
	side. We prove convergence for our methods and compare their numerical 
	performance. We show how adding further information about call option 
	prices at additional maturities can be incorporated and narrows down the 
	no-arbitrage pricing bounds. Finally, we obtain structural results for the 
	case of the payoff given by a weighted sum of covariances between the 
	assets. 
\end{abstract}



\section{Introduction}
Mathematical modelling is a ubiquitous aspect of modern the financial industry 
and 
it drives important decision processes. Stochastic models are a key component 
used to describe evolution of risky assets and quantify financial risks. The 
ability to postulate and analyse such models was at the heart of the growth in 
ever more complex derivatives trading and other aspects of the financial 
markets.    
However, understanding well the implications of a given model is not 
sufficient. Equally important is to appreciate the consequences of the model 
being wrong in the sense of being an inadequate or misguided description of the 
reality. The latter issue is often referred to as the \emph{Knightian 
	uncertainty} after \cite{Knight:21}. 
This dichotomy between risk and uncertainty, and the quest to capture both and 
understand their interplay, are at the heart of the field of Robust 
Mathematical Finance. The field is concerned 
with the modelling space, from model-free to model-specific approaches, and 
with understanding and quantifying the impact of making assumptions, and of using or ignoring 
market information. It has been an important area of research, in particular in 
the last decade following the financial crisis, and we refer to \cite{bfhmo} 
and the references therein for an extensive discussion. One of the most active 
research topics within the field has been that of model-independent pricing and 
hedging of derivatives. It goes back to \cite{hobson1998robust} and 
probabilistic methods related to Skorokhod embeddings, see for example 
\cite{brown2001robust,cox2011robusta}. More recently, it has been recast as an 
optimal transport problem with a martingale constraint, see 
\cite{BHLP,galichon2014stochastic} and gained a novel momentum. A significant 
body of research grew studying this \emph{Martingale Optimal Transport} (MOT) 
problem both in discrete and continuous time, see for example 
\cite{BCH:17,BNT:17,DolinskySoner:2014,hou2015robust} and the references 
therein. More recently, first numerical methods for MOT problems were developed 
in \cite{GO,Eckstein2019}. However, all these works assume that markets provide 
sufficient information to derive the joint, multi-dimensional, risk neutral distribution of assets at given maturities. In dimensions greater than one, 
this assumption is unrealistic in most markets.

In contrast, in this paper, we propose to study problems, in dimensions greater 
than one, which are directly motivated by typical market settings and the 
available market data. Our focus is on numerical methods and we aim to deliver 
a proof-of-concept results which, we hope, could spark interest in these 
methods among industry practitioners. 
More precisely, we assume market prices of call and put options are given for 
individual assets - these could be for one or many maturities. For simplicity, 
we focus on the case when such prices are given for enough strikes to derive 
the implied risk-neutral distribution, a standard argument going back to 
\cite{BreedenLitzenberger:78}. Our numerical methods can easily be adjusted to 
the case of only finitely many traded call options and we establish a 
continuity result to justify our focus on the synthetic limiting case. Given 
the market information, we study the implied no-arbitrage bounds for an option 
with a payoff which depends on multiple assets. A simple example, with two 
assets, is given by a spread option. In higher dimensions, natural examples are 
given by options written on an index. We stress that while market information 
translates into risk neutral distributional constraints on individual assets, 
the global no-arbitrage constraint translates into a global martingale 
constraint which binds the assets together and is sharper than just requiring 
that each of the assets were a martingale in its own filtration. 

We call the resulting optimisation problem a \emph{Multi-Marginal Martingale 
	Optimal Transport} (MMOT) problem. It was first studied in \cite{Lim} who 
focused on its duality theory. The duality is of intrinsic financial interest: 
while the primal problem corresponds to the risk-neutral pricing, the dual side 
corresponds to optimising over hedging strategies. The equality between the 
primal and the dual problem corresponds to the superhedging duality in 
mathematical finance. We exploit it here to propose two different numerical 
methods for MMOT problems. First, we adopt the approach of \cite{GO} and 
propose a computational method for the primal problem. This relies on 
discretisation of the marginal measures combined with a relaxation of the 
martingale condition. Theorem \ref{thm:mot-lp} establishes convergence of the 
approximating problems to the original MMOT problem. Each approximating problem 
in turn, is a discrete LP problem and can be solved efficiently. The main 
disadvantage of this approach is the curse of dimensionality: LP problems with 
too many constraints quickly exceed memory capacity. Our second approach builds 
on the work of \cite{Eckstein2019} to develop a computational method for 
solving the dual problem. The dual problem involves an optimisation over 
hedging strategies and we approximate these with elements of a deep neural 
network (NN). To employ the stochastic gradient descent we change the problem 
from a singular one, with the superhedging inequality constraint, to a smooth 
one with an integral penalty term. Theorem \ref{thm:cnv_nn} shows that under 
suitable assumptions the results converge, with the penalty term $\gamma\to 
\infty$ and the size of the NN  $m\to \infty$, to the value of the MMOT 
problem. 

The numerical methods we propose, including Theorems \ref{thm:mot-lp} and \ref{thm:cnv_nn}, are natural extensions of the previous MOT studies cited above and do not require fundamental new insights. In this way, our results illustrate that these studies are generic in nature and can be extended or generalised to  other similar contexts. Apart from the MMOT problem studied here, we mention the optimization problem considered by \cite{KramkovXu:19}. A large focus of our paper then lies on making the method practically applicable and testing their numerical performance. 
We discuss the details of the implementation and provide GitHub links with the codes.
We show that the NN approach agrees with the LP approach but is also able to 
handle higher dimensional settings. Both approaches are shown to recover 
theoretical values, when these can be computed independently. By focusing on a 
MMOT problem which corresponds to real world industry scenarios we hope to 
showcase the capacity of the robust approach to capture and quantify, in a 
fully non-parametric and model-agnostic way, the impact of various sources of 
information, or ways to trade, for a given pricing and hedging problem. We 
illustrate this on both synthetic and real world data. 
In particular, we consider the case of payoffs only depending on the assets' 
terminal values at time $T$, e.g., spread options and options paying covariance 
between the assets. Such examples allow us to capture the value of additional 
market information from intermediate maturities. Indeed, we can start by considering only the call prices at 
time $T$, i.e., MMOT becomes just an optimal transport problem, or the so-called robust copula, see for example \cite{Wang2013}. 
Adding call prices at earlier maturities $T_i<T$ then reduces the range of no-arbitrage 
prices and thus captures the value of this information for robust pricing and 
hedging. This, along with the structure of optimisers, can be understood and 
characterised theoretically as Theorem \ref{thm:structure} shows. 

The remainder of this paper is structured as follows. In Section \ref{sec:mmot} 
we introduce the MMOT problem and its duality. Then we develop our 
computational methods: first the LP approach in Section \ref{sec:LPtheory} and 
then the NN approach in Section \ref{sec:NNtheory}. All the numerical examples 
are presented in the subsequent Section \ref{sec:num_ex}. Finally, in Section 
\ref{sec:structure}, we discuss some structural results for the particular case 
of the covariance payoff.

\section{The MMOT problem}\label{sec:mmot}
We denote by $\Pc(\R^d)$ the set of probability measures on $\R^d$ with a finite first moment. 
Measurable (resp.\ continuous) functions from $\R^d$ to $\R^k$ are denoted $\Lsp^0(\R^d;\R^k)$ (resp.\ 
$C(\R^d;\R^k)$) and we write $C_b$ for continuous bounded functions. For a 
$\mu\in \Pc(\R^d)$, $\Lsp^1(\mu)$ denotes the space of functions $f : \R^d \to \R$ with $\int |f| d\mu < \infty$.

Let $T, d \in \mathbb{N}$ and $\Xc_1 = \Xc_2 = ... = \Xc_T = \mathbb{R}^d$, 
$\Xc = \Xc_1 \times ... \times \Xc_T$. We denote the natural projection from 
$\Xc$ onto its $t^\textrm{th}$ component by $X_t$, and by $X_{t, i}$ the further 
projection onto the $i$-th component of $\Xc_t$. For $x \in \Xc$ we write $x_t 
= X_t(x)$ and $x_{t, i} = X_{t, i}(x)$.
Given $\m_{t, i} \in \Pc(\mathbb{R})$, let $\mv_t = (\m_{t, i})_{1 \leq i \leq 
	d}$ and $\mv=(\mv_t)_{1\leq t\leq T}$. We define $\Pi(\mv) = \Pi(\mv_1, 
	..., 
\mv_T) \subset \Pc(\Xc)$ as the set of measures $\pi$ satisfying $\pi \circ 
X_{t, i}^{-1} = \mu_{t, i}$ for $1\leq t\leq T$ and $1\leq i\leq d$. We stress that throughout we only consider measures with a finite first moment.

We assume from now onwards that $\mu_{t,i}$ are increasing in convex order in $t$, denoted $\mu_{t, i} \leqcvx \mu_{t+1, i}$, by which we mean that
$$\int f d\mu_{t, i} \le \int f d\mu_{t+1, i}\quad \text{for all convex functions }f,\;\; 1\leq t\leq T-1, 1\leq i\leq d.$$
We define $\Mc(\mv)=\Mc(\mv_1, ..., \mv_T) \subseteq \Pi(\mv)$ to be the subset 
consisting of martingale measures, i.e., measures $\pi$ such that 
\[
\mathbb{E}_\pi[X_{t+1}|X_1, ..., X_t] = X_t,\quad 1\leq t\leq T-1.
\]
It follows from \cite{strassen1965existence} that our increasing convex order assumptions on 
$\mv$ are precisely the necessary and sufficient conditions for $\Mc(\mv)\neq 
\emptyset$.

Our object of interest in this paper is the \emph{multi-marginal martingale optimal transport} (MMOT) defined as 
\begin{equation}
\begin{split}
\mmoto(\mv) \q &:= \q \po(\mv) \q := \q \sup_{\pi\in \Mc(\mv)}~ \int c \,d\pi, 
\label{def:mmot} \\
\mmotu(\mv) \q &:= \q \pu(\mv) \q := \q \inf_{\pi\in \Mc(\mv)}~ \int c \,d\pi,
\end{split}
\end{equation}
for a given measurable function $c:\Xc\to\R$ to optimize. We recall that the 
martingale condition encodes the financial requirement of absence of arbitrage. 
We mostly use the notation $\po, \pu$. However when we want to stress the 
martingale condition, we write $\mmoto, \mmotu$. Without this condition, the 
problem above corresponds to the multi-marginal optimal transport, given by
\begin{equation}
\begin{split}
\oo(\mv) \q &:= \q \sup_{\pi\in \Pi(\mv)}~ \int c \,d\pi, \label{def:ot} \\
\ou(\mv) \q &:= \q \inf_{\pi\in \Pi(\mv)}~ \int c \,d\pi. 
\end{split}
\end{equation}
In the particular case when $d=2$ and $c(x)=c(x_{T,1},x_{T,2})$ the above 
corresponds to the classical optimal transport problem on $\R$ as only the 
marginals $\mu_{T,i}$, $i=1,2$, impact the problem. The case 
$c(x)=|x_{T,1}-x_{T,2}|$ gives $\ou(\mv)= \Wc_1(\mu_{T,1}, \mu_{T,2})$, which 
is the Wasserstein distance of order $1$, a metric on $\Pc(\R)$ which we will 
use extensively in Section \ref{sec:LPtheory}. Note that in general \[\ou 
\leqslant  \mmotu  \leqslant  \mmoto \leqslant  \oo.\] Both problems 
\eqref{def:mmot} and \eqref{def:ot} admit a dual formulation. The latter can be 
found in \cite{bartl2017duality}, while the former was developed in \cite{Lim}, 
following earlier works on the martingale optimal transport in \cite{BHLP}. We 
recall it here as it will be used for our numerical methods. Define 
 $\Do$, respectively $\Du$, to be the set of functions $(\varphi_{t, 
i})_{1\leq t \leq T, 1 \leq i \leq d}$ and $(h_{t, i})_{1\leq t \leq T-1, 1 \leq i \leq d}$, where 
$\varphi_{t, i} \in \Lsp^1(\mu_{t, i})$ and $h_{t, i} \in \Lsp^0(\mathbb{R}^{t 
	\cdot d};\mathbb{R})$, satisfying for all $x \in \Xc$:
\b*
\sum_{t=1}^T \sum_{i=1}^d \varphi_{t, i}(x_{t, i}) + \sum_{t=1}^{T-1} 
\sum_{i=1}^{d} h_{t, i}(x_1, ..., x_t) (x_{t+1, i} - x_{t, i}) &\ge& c(x), 
\quad \text{respectively} \\
\sum_{t=1}^T \sum_{i=1}^d \varphi_{t, i}(x_{t, i}) + \sum_{t=1}^{T-1} 
\sum_{i=1}^{d} h_{t, i}(x_1, ..., x_t) (x_{t+1, i} - x_{t, i}) &\le& c(x).
\e*
Then the corresponding dual problems are defined by
\begin{equation}
\begin{split}
\ddo(\mv) \q :=& \q \inf_{(\vp_{t, i},h_{t, i}) \in\Do}~ 
\sum_{t=1}^T\sum_{i=1}^d \int \vp_{t, i} \,d\mu_{t, i}, \label{def:mmotdual} \\
\du(\mv) \q :=& \q \sup_{(\vp_{t, i},h_{t, i}) \in\Du}~ \sum_{t=1}^T 
\sum_{i=1}^d \int \vp_{t, i} \,d\mu_{t, i}.
\end{split}
\end{equation}
Let us now discuss briefly the financial interpretation of the objects introduced so far. We refer the reader to \cite{BHLP,bfhmo} and the references therein for more details and for background on robust financial mathematics. We consider a market with $d$ traded risky assets and $T$ time steps, or maturities. All prices are discounted, i.e., given in units of a fixed numeraire (e.g., the bank account). The price of $i^\textrm{th}$ asset at $t^\mathrm{th}$ maturity is denoted $x_{t,i}$ above. We suppose market provides call/put option prices for each of these assets for 
all $T$ maturities and all strikes. This, through a classical argument of 
\cite{BreedenLitzenberger:78}, is equivalent to fixing the risk neutral 
marginal distributions of the assets. We denote $\mu_{t,i}$ the risk-netural 
marginal distribution of $i^\textrm{th}$ asset at $t^\mathrm{th}$ maturity. In 
particular, the first maturity will often correspond to today and then   
$\mu_{1,i}=\delta_{s_{i}}$ where $s_i$ is today's price of the  
$i^{\textrm{th}}$ asset. $\Mc(\mv)$ is the set of all risk neutral (i.e., 
martingale) measures for the whole market which are calibrated to the given 
option prices. Thus, the primal problem \eqref{def:mmot} takes the risk-netural 
pricing perspective and gives the upper and lower bounds for the price of a 
derivative with payoff $c$ at the final maturity. The dual problem 
\eqref{def:mmotdual} considers the same quantities but from the hedging 
perspective. Here, $\varphi_{t, i}(x_{t, i})$ represents the static position 
synthetised from $t^\textrm{th}$ maturity call/put options on the 
$i^\textrm{th}$ asset and 
$h_{t,i}$ is the number of shares of the $i^\textrm{th}$ asset held between the $t^\textrm{th}$ and $(t+1)^\textrm{th}$ maturity. Importantly, $h_{t,i}$ is a function of the past prices of \emph{all} assets across the previous maturities. This, on the primal side, corresponds to the requirement that the assets are jointly martingale and not just each on their own. Classically, the pricing and the hedging approach should give the same no-arbitrage price range. This is also the case here as stated in the following theorem.
\begin{theorem}\label{thm:duality}
	Let $\mv \in \Pc(\R)^{dT}$ with $\Mc(\mv)\neq\emptyset$, and let $\psi : 
	\Xc \rightarrow \mathbb{R}$ be given by 
	$\psi(x)=1+\sum_{t=1}^T\sum_{i=1}^d |x_{t,i}|$. 
	If $c:\Xc\to\R$ is lower semi-continuous and $ c\geq -K\psi$ on $\Xc$ for 
	some $K>0$ then $\pu(\mv)=\du(\mv)$. If $c:\Xc\to\R$ is upper 
	semi-continuous and $c\leq K\psi$ on $\Xc$ for some $K>0$ then 
	$\po(\mv)=\ddo(\mv)$. In both cases, the primal problems are attained and 
	the dual values remain unchanged when one restricts to $\varphi_{t, i} \in 
	\Lsp^1(\mu_{t, i}) \cap C(\R;\R)$, $h_{t, i} \in C_b(\mathbb{R}^{t \cdot 
		d};\mathbb{R}^d)$, $1\leq t \leq T$, $1\leq i\leq d$.
\end{theorem}
We note that this result was proved in \cite{Zaev} with the assumption of 
continuous cost $c$, but it is standard to extend the duality to the 
semi-continuous costs, see, e.g., \cite{Vi03, Vi09}. We also note that this 
duality for martingale optimal transport was first proved in \cite{BHLP} in one 
dimension $d=1$, and they also showed that the duality holds with a narrower 
class of functions $\varphi_{t,i}$ which are linear combinations of finitely 
many call options, i.e. $\vp_{t,i}$ of the form $c_{t,i}+\sum_{j=i}^{l_{t,i}} 
c_{t,i,j} (x_{t,i}-k_{t,i,j})^+$, for some $l_{\cdot}, c_{\cdot}, k_{\cdot}$. 
The same applies here.

While existence of primal optimizers in \eqref{def:mmot} is easy to obtain, in 
general we cannot hope for uniqueness. We illustrate this with two simple 
examples. In both examples, $d=2=T$ and $c(x)=x_{2,1}x_{2,2}$. We further study 
this particular cost function and present some structural results in Section 
\ref{sec:structure}.

\begin{example}\label{Ex1}
	Consider $d=2=T$ and the maximization problem with $c(x)=x_{2,1}x_{2,2}$. 
	Take $\mu \leqcvx \nu$ such that $\Mc(\mu,\nu)$ is not a singleton, e.g., $\mu,\nu$ are Gaussians with the same mean and increasing variance, and let 
	$\mu_{1,1}=\mu_{1,2}=\mu$, $\mu_{2,1}=\mu_{2,2}=\nu$. Then for any 
	$\tilde\pi \in \Mc(\mu,\nu)$, the distribution $\pi$ of any quadruple of random variables
	$(\xi,\xi,\eta,\eta)$ satisfying $(\xi,\eta)\sim \tilde \pi$ is an element of 
	$\Mc(\mv)$. Further, $\pi \circ X_{2}^{-1}$ is the monotone increasing 
	coupling of $\nu$ with itself and, in particular, is independent of the choice of $\tilde \pi$. It also  
	 attains $\po(\mv)$ and we conclude that the optimizer in $\po(\mv)$ is 
	not unique. Note however that, in this example, the distributions $\pi\circ 
	X_1^{-1}$ and $\pi\circ X_2^{-1}$ are the same for any optimizer $\pi\in 
	\Mc(\mv)$. 
\end{example}

\begin{example} \label{Ex2}
 Consider the same problem as in Example \ref{Ex1} but with 
	$\mu_{1,1} = \delta_0$, $\mu_{2,1} = 
	\frac{1}{4}(\delta_{-2}+\delta_{-1}+\delta_1+\delta_2)$, 
	$\mu_{1,2}=\mu_{2,2} = 
	\frac{1}{2}(\delta_{-1}+\delta_1)$. Note that, for any $\pi\in \Mc(\mv)$, 
	$\pi^1= \pi \circ X_1^{-1} = \frac{1}{2}(\delta_{(0,1)} + \delta_{(0, -1)})$. 
	Further, the following measures dominate $\pi^1$ in the convex order and 
	have 
	$\mu_{2,1}, \mu_{2,2}$ as their marginals:
	\begin{align}
	\pi^2= \frac{1}{4}(\delta_{(-1,1)}+\delta_{(1,1)}+\delta_{(-2,-1)} 
	+\delta_{(2,-1)}), \nonumber \\
	\tilde \pi^2 = \frac{1}{4}(\delta_{(-1,-1)}+\delta_{(1,-1)}+\delta_{(-2,1)} 
	+\delta_{(2,1)}). \nonumber
	\end{align}
	Hence there exist $\pi, \tilde \pi\in \Mc(\mv)$ whose (2-dimensional) marginals are 
	$\pi^1, \pi^2$ and $\pi^1, \tilde \pi^2$ respectively. In particular, 
	$\Mc(\mv)$ is not a singleton. However, for any $\pi\in \Mc(\mv)$, we have
	$$\E_\pi[X_{2,1}X_{2,2}]=\E_{\pi}[X_{2,1}X_{1,2}]=\E_{\pi}[X_{1,1}X_{1,2}]=0,$$
	and hence $\pi$ is an optimizer for both $\pu(\mv)$ and $\po(\mv)$ with 
	$c(x)=x_{2,1}x_{2,2}$. In this example, neither the optimizer nor the 
	implied distribution of $X_2$ are unique. 
\end{example}

\section{Numerical Methods for MMOT problems}\label{sec:num_methods}

We present now two numerical approaches for computing the MMOT value 
\eqref{def:mmot}, as well as the primal and the dual optimizers. Our first 
approach relies on the primal formulation \eqref{def:mmot} and LP methods. 
Our second approach starts with the dual formulation \eqref{def:mmotdual}, uses penalization to convexify the superhedging constraint and employs optimization techniques involving deep neural networks. 

Before we discuss our methods in detail, we mention briefly two other numerical approaches which have been applied to MOT problems and, while not pursued in this paper, could potentially be extended to the MMOT context. The first one is the cutting plane method as described in \cite{phlautomated}. This method is LP based but works via the dual side. A standard LP approach runs into memory issues for higher values of $d$ or $T$, see section \ref{subsec:accuracy}, and the cutting plane method circumvents this by considering a finite set of basis functions. While effective, this introduces a new source of error and one which is difficult to control theoretically and for this reason, see also section \ref{subsec:Num1}, we do not discuss it further in this paper.
The second method is a generalization of the Sinkhorn algorithm \cite{benamou2015iterative,cuturi2013sinkhorn} to the MOT problem, see \cite{de2018entropic}. The Sinkhorn algorithm adds a strictly convex/concave term to the objective function, similarly as our neural network approach presented below. However, the Sinkhorn algorithm is designed for discrete marginals and based on the so-called iterative Bregman projections. In the OT context, the projections are onto the marginal constraints. The added difficulty for MOT problems comes from the additional projection onto the martingale constraint, which does not admit a closed form solution. Instead, the projection has to be approximated numerically, e.g., using Newton's method, which again introduces a new source of error. To the best of our knowledge, even for MOT problems, there are no theoretical results ensuring that the accumulative error does not explode and the number of required projections does not increase with respect to the discretisation precision.

\subsection{The Primal Problem - an LP approach}\label{sec:LPtheory}
Following the approach of \cite{GO}, we propose a computational scheme to solve 
\eqref{def:mmot}. For each $\eps\in\R_+$, denote by $\Mc_{\eps}(\mv)\subset 
\Pi(\mv)$  the subset of measures $\pi$ satisfying
\b*
\E_{\pi}\Big[ \Big| \E_{\pi}\big[X_{t+1}\big|X_1,\ldots, 
X_t\big]~-~X_t\Big|\Big] \q \le \q \eps, \q   \mbox{ for }t=1,\ldots, T-1,
\e*
where $|\cdot|$ stands for the $\ell_1$ norm. Introduce, accordingly, the 
optimization problems as follows:
\[
\po_{\eps}(\mv)\q := \q \sup_{\pi\in \Mc_{\eps}(\mv)}~ \int c \,d\pi, \qquad
\pu_{\eps}(\mv)\q := \q \inf_{\pi\in \Mc_{\eps}(\mv)}~ \int c \,d\pi.
\]
Then clearly $\po_{0}= \po$ (resp. $\pu_{0}= \pu$), and Theorem 
\ref{thm:mot-lp} provides the basis of our numerical method.

\begin{theorem}\label{thm:mot-lp}
	Let $\mv\in  \Pc(\R)^{dT}$ satisfy $\Mc(\mv)\neq\emptyset$ and 
	$\mv^n=\big(\mu_{t,i}^n\big)_{1\leq t\leq T, 1\leq i\leq d}$ satisfy 
	$\lim_{n\to\infty} r_n=0$ with $r_n:=2\max_{1\le t\le T} \sum_{1\leq i \leq 
	d} \Wc_1(\m_{t, i}^n, \m_{t, i})$. Then, for all $n\ge 1$, 
	$\Mc_{r_n}(\mv^n)\neq\emptyset$. Assume further $c$ is Lipschitz, then:
	
	\vspace{1mm}
	
	\no \rmi For any $(\eps_n)_{n\ge 1}$ converging to zero such that 
	$\eps_n\ge r_n$ for all $n\ge 1$, one has 
	\b*
	\lim_{n\to\infty} \po_{\eps_n}(\mv^n)  = \po(\mv) \mbox{ and }
	\lim_{n\to\infty}\pu_{\eps_n}(\mv^n) = \pu(\mv).
	\e*
	\rmii For each $n\ge 1$, $\po_{\eps_n}(\mv^n)$  (resp. 
	$\pu_{\eps_n}(\mv^n)$) admits an optimizer $\pi_n$. The sequence 
	$(\pi_n)_{n\ge 1}$ is tight and every limit point is an optimizer for 
	$\po(\mv)$  (resp. $\pu(\mv)$). In particular, $(\pi_n)_{n\ge 1}$ converges 
	weakly whenever $\po(\mv)$  (resp. $\pu(\mv)$) has a unique optimizer.
\end{theorem}
Before proving this result we state and prove two preliminary propositions.

\begin{proposition}\label{prop:approx} 
	Provided $\mv \in \Pc(\R)^{dT}$ with $\Mc_{\eps}(\mv)\neq\emptyset$, it 
	holds $\Mc_{\eps+r}(\mv')\neq\emptyset$ for all $\mv'\in \Pc(\R)^{dT}$ 
	where $r:=2\max_{1\le t\le T} \sum_{1\le i\le d} \Wc_1(\m_{t, i}', \m_{t, 
	i})$. Further, if $c$ is  $L-$Lipschitz, then 
	\b*
	\po_{\eps}(\mv)  ~\le~ \po_{\eps+r}(\mv') + LTr/2 \quad \text{and} \quad 
	\pu_{\eps}(\mv)  ~\ge~ \pu_{\eps+r}(\mv') - LTr/2.
	\e*
\end{proposition}

\begin{proof}
	Without loss of generality, we only show the first inequality. Fix an 
	arbitrary $\pi\in \Mc_{\eps}(\mv)$. It follows from Skorokhod's theorem 
	that, there exists an enlarged probability space $(E , \Ec , \Qc)$ which 
	supports random variables $U_t=(U_{t,1},\ldots, U_{t, d})$, 
	$Z_t=(Z_{t,1},\ldots, Z_{t,d})$ taking values in $\R^d$, for $t=1,\ldots, 
	T$, such that
	\begin{align}\label{eq:Skoro_conditions}
	\begin{split}
	~~~ &\text{$\bullet$~$\Qc\circ (U_1, ..., U_T)^{-1} =\pi$ and $\Qc\circ 
	Z_t^{-1} =  \Nc_d$ for $t=1,\ldots, T$,} \\ 
	~~~ &\text{where $\Nc_d$ denotes the standard normal distribution on 
	$\R^d$.} \\
	~~~ &\text{$\bullet$~$(U_1,\ldots, U_T)$ and $(Z_1,\ldots, Z_T)$ are 
	independent.}
	\end{split}
	\end{align}
	Let $t \in \{ 1, \ldots, T-1\}$. For $i=1,\ldots, d$, let $\gamma_{t, i}$ 
	be the optimal transport plan realizing the Wasserstein distance 
	$\Wc_1(\m_{t, i},\m_{t, i}')$. Using standard disintegration techniques 
	(see \cite[Lemma A.1]{GO}), there exist measurable functions $f_{t,i}: 
	\R^2\to\R$ such that $\Qc\circ (U_{t,i}, V_{t,i})^{-1}= \gamma_{t,i}$ with  
	$V_{t,i}:= f_{t,i}(U_{t,i}, Z_{t,i})$.  Let $V_t:=(V_{t,1},\ldots, 
	V_{t,d})$. Then, for all $h=(h_i)_{1\le i\le 
	d}\in\Cc_{b}\big(\Xc_1\times\cdots\times \Xc_t; \R^d\big)$, one has
	\b*
	&&\E_{\Qc}\big[h(V_1,\ldots, V_t) \cdot 
	(V_{t+1}-V_t)\big]~=~\E_{\Qc}\left[\sum_{i=1}^d h_i(V_1,\ldots, 
	V_t)(V_{t+1,i}-V_{t,i})\right] \\
	&&=~ \E_{\Qc}\left[\sum_{i=1}^d h_i(V_1,\ldots, 
	V_t)(V_{t+1,i}-U_{t+1,i})\right]~+~\E_{\Qc}\left[\sum_{i=1}^d 
	h_i(V_1,\ldots, V_t)(U_{t+1,i}-U_{t,i})\right] \\
	&&~~+~  \E_{\Qc}\left[\sum_{i=1}^d h_i(V_1,\ldots, 
	V_t)(U_{t,i}-V_{t,i})\right] \\
	&&\le ~ r\|h\|_{\infty}~+~\E_{\Qc}\left[\sum_{i=1}^d 
	h_i\big(f_{s,i}(U_{s,i}, Z_{s,i});1\le s\le t, 1\le i\le d\big) 
	(U_{t+1,i}-U_{t,i})\right] \\
	&&\le~ (\eps+r)\|h\|_{\infty},
	\e*
	where the last inequality follows from the conditions in 
	\eqref{eq:Skoro_conditions}. Therefore, 
	\be\label{def:monoto}
	\int h(x_1,\ldots, x_t)\cdot (x_{t+1}-x_t)\pi'(dx) \q \le\q 
	(\eps+r)\|h\|_{\infty}
	\ee
	holds for all $h\in\Cc_{b}\big(\Xc_1\times\cdots\times\Xc_t;\R^d\big)$, 
	where $\pi':=\Qc\circ (V_1,\ldots, V_T)^{-1}$. In view of the monotone 
	class theorem, this is equivalent to  
	\b*
	\E_{\pi'}\Big[ \Big| \E_{\pi'}\big[X_{t+1}\big|X_1,\ldots, 
	X_t\big]~-~X_t\Big|\Big] \q \le \q \eps+r.
	\e*
	Hence, $\pi'\in \Mc_{\eps+r}(\mv')\neq\emptyset$  as $\pi'\circ 
	X_{t,i}^{-1}=\m_{t,i}'$ for $t=1,\ldots, T$ and $i=1,\ldots, d$. To 
	conclude the proof, notice that 
	\b*
	\int cd\pi~-~\po_{\eps+r}(\mv') &\le&  \int cd\pi~-~\int cd\pi' ~~=~~ \E_{\Qc}\big[c(U_1,\ldots, 
	U_T)-c(V_1,\ldots, V_T)\big] \\
	&\le& L\sum_{t=1}^T\sum_{i=1}^d \E_{\Qc} \big[|U_{t,i}-V_{t,i}|\big] 
	~~\le~~ LTr/2,
	\e*
	which yields $\po_{\eps}(\mv)\le \po_{\eps+r}(\mv')+LTr/2$ 
	as $\pi\in \Mc_{\eps}(\mv)$ is arbitrary. 
\end{proof}

\begin{proposition}\label{prop:continuity}
	Assume that $c$ has a linear growth and $\Mc(\mv)\neq\emptyset$.
	
	\vspace{1mm}
	
	\no \rmi If $c$ is u.s.c., then the map $
	\R_+\ni\eps\mapsto \po_{\eps}(\mv)\in \R$ 
	is non-decreasing, continuous and concave.
	
	\vspace{1mm}
	
	\no \rmii If $c$ is l.s.c., then the map $
	\R_+\ni\eps\mapsto \pu_{\eps}(\mv)\in \R$ 
	is non-increasing, continuous and convex.
\end{proposition}

\begin{proof}
	We only show {\rm (i)} here. First notice that $\eps\mapsto  
	\po_{\eps}(\mv)$ is non-decreasing by definition. Next, let us prove the 
	concavity. Given $\eps$, $\eps'\in\R_+$ and $\alpha\in [0,1]$, it remains 
	to show $(1-\alpha) \po_{\eps}(\mv)+\alpha \po_{\eps'}(\mv) \le 
	\po_{\eps_{\alpha}}(\mv)$, where  
	$\eps_{\alpha}:=(1-\alpha)\eps+\alpha\eps'$. This indeed follows from the 
	fact that $(1-\alpha)\pi+\alpha\pi'\in \Mc_{\eps_{\alpha}}(\mv)$ for all 
	$\pi\in \Mc_{\eps}(\mv)$ and $\pi'\in \Mc_{\eps'}(\mv)$. Hence the map 
	restricted to $(0,+\infty)$ is continuous. Finally, let us show the right 
	continuity at zero. For any sequence $(\eps_n)_{n\ge 1}\subset\R_+$ 
	decreasing to zero, let $(\pi_n)_{n\ge 1}$ be a sequence such that 
	$\pi_n\in \Mc_{\eps_n}(\mv)$ for $n\ge 1$ and $\lim_{n\to\infty}  
	\po_{\eps_n}(\mv)  = \lim_{n\to\infty} \int c d\pi_n$. {We have 
	    \b* 
    \lim_{R\to\infty} \sup_{n\ge 1}\pi_n\left(\big(\R \setminus [-R,R]\big)^{Td}\right) \leq \lim_{R\to\infty} Td\left(\sup_{t\leq T, i\leq d}\mu_{t,i}\big(\R \setminus [-R,R]\big)\right) &=& 0,
    \e* 
    which shows that $(\pi_n)_{n\ge 1}$ is tight and hence, by Prokhorov's theorem, admits a 
	weakly convergent subsequence $(\pi_{n_k})_{k\ge 1}$. As the marginals are fixed and have finite first moments, we see that the convergence holds in $\Wc_1$ and that the limit $\pi\in 
	\Mc(\mv)$.
	This implies, thanks to our assumptions on $c$, that  }
	\b*
	\lim_{n\to\infty}\po_{\eps_n}(\mv)~ =~ 
	\lim_{k\to\infty}\po_{\eps_{n_k}}(\mv)~ \leq ~  \po(\mv).
	\e*
	Combined with the obvious reverse inequality this yields the right 
	continuity at zero. 
\end{proof}

\begin{proof}[Proof of Theorem \ref{thm:mot-lp}]
	\rmi It suffices to deal with the maximization problem. First, by 
	Proposition 
	\ref{prop:approx}, we have $\emptyset \neq\Mc_{r_n}(\mv^n)\sb \Mc_{\eps_n}(\mv^n)$ and further
	\b*
	\po(\mv) ~\le~ \po_{r_n}(\mv^n)+LT r_n/2 ~\le~ \po_{\eps_n}(\mv^n) + 
	LT\eps_n/2,
	\e*
	where $L$ denotes the Lipschitz constant of $c$.  Repeating the above 
	reasoning  but interchanging $\mv$ and $\mv^n$, we obtain $
	\po_{\eps_n}(\mv^n) \le \po_{2\eps_n}(\mv) + LT\eps_n/2$, which yields 
	finally
	\b*
	-LT\eps_n/2  \le \po_{\eps_n}(\mv^n)- \po(\mv) \le 
	\left(\po_{2\eps_n}(\mv)-\po(\mv)\right) + LT\eps_n/2.
	\e*
	This result then follows by Proposition \ref{prop:continuity}.
	
	\vspace{1mm}
	
	\no\rmii Arguments in the proof of Proposition \ref{prop:continuity} above show that $\Mc_{\eps_n}(\mv^n)$ is compact. Combined with the Lipschitz continuity of $c$, this yields the existence of $\pi_n$. To show tightness of $(\pi_n)_{n\ge 1}$, let $\eps >0$ and observe that $r_n\to 0$ implies that $\mu^n_{t,i} \to \mu_{t,i}$ in $\Wc_1$ for all $t\leq T$, $i\leq d$ and hence there exists $N$ such that for every $n \ge N$,
	\begin{align*}
	\pi_n\left(\big(\R \setminus [-R,R]\big)^{Td}\right) &\le \sum_{t \le T, i \le d} \mu^n_{t,i}\big(\R \setminus [-R,R]\big) \\
	& \le \sum_{t \le T, i \le d} \mu_{t,i}\big(\R \setminus [-R+1,R-1]\big) + \eps/2. 	\end{align*}
Hence, we can take $R_\eps$ large enough so that 
		\begin{align*}
	\sup_{n \ge 1} \pi_n\left(\big(\R \setminus 
	[-R_\eps,R_\eps]\big)^{Td}\right) \le \eps.
	\end{align*}
	Thus $(\pi_n)_{n\ge 1}$ is tight and hence, by Prokhorov's theorem, admits 
	a weakly convergent subsequence $(\pi_{n_k})_{k\ge 1}$ with a limit denoted by 
	$\pi$. Further, again since $r_n\to 0$, the first moments converge so that $\pi_n\to \pi$ in $\Wc_1$. 
	Using the alternative  definition \eqref{def:monoto} and the 
	dominated convergence theorem, we see that $\pi\in \Mc(\mv)$.  
\end{proof}

The above discussion and Theorem \ref{thm:mot-lp} rely on having a sequence of 
discrete measures $\mv^n=\big(\mv_{t,i}^n\big)$ converging to $\mv$. As each 
$\mu_{t,i}$ is a probability measure on $\R$, its discretisation is a well 
studied subject. For the sake of simplicity, we write $\mu\equiv \mu_{t,i}$ in 
the rest of this section. Suppose first that $\mu$ is given via its density or 
its CDF, or an equivalent functional representation. We could then follow the 
abstract approach in \cite[Section 3.1]{GO}, noting that for $d=1$ the first 
step (\emph{Truncation}) can be simplified to take $\mu_R(dx) := 
\mathds{1}_{B_R}(x)\mu(dx)/\mu[B_R]$, where $B_R=[-R,R]$. 

However, more explicit methods are possible. One such discretisation was 
proposed in \cite{DolinskySoner:2014} and corresponds to taking $\mu^n$ 
supported on $\{k/n\}_{k\in\Z}$: 
\begin{equation}\label{eq:DSdiscretisation}
\mu^n\left(\left\{\frac{k}{n}\right\}\right):= \int_{[(k-1)/n, (k+1)/n)} 
\left(1-|nx-k|\right)\mu(dx),\quad k\in \Z. 
\end{equation}
The construction has a natural interpretation in the potential-theoretic 
language, see \cite{Chacon:77}, namely $\mu^n$ is the probability measure whose 
potential agrees with that of $\mu$ on $\{k/n\}_{k\in\Z}$ and is linear 
otherwise. This implies, in particular, that 
the discretisation preserves the convex order: if $\mu\leqcvx\nu$ then 
$\mu^n\leqcvx\nu^n$. Note also that  for any measurable function $f:\R\to\R$, 
it holds
\b*
\int_{\R}f(x)\mu^n(dx) &=&\int_{\R}f_n(x)\mu(dx),
\e* 
where $f_n(x):=(1+\lfloor nx\rfloor -nx) f\left(\lfloor 
nx\rfloor/n\right)+(nx-\lfloor nx\rfloor) f\left((1+\lfloor 
nx\rfloor)/n\right)$. One has thus by the dual formulation that 
$\Wc_1(\mu^n,\mu)\le 1/n$. Further, a straightforward computation yields 
\b*
&&\int_{[(k-1)/n, (k+1)/n)}\left(1-|nx-k|\right)\mu(dx)\\ 
&&= 
n\int_{\R}\left(\left(x-\frac{k-1}{n}\right)^++\left(x-\frac{k+1}{n}\right)^+-2\left(x-\frac{k}{n}\right)^+\right)\mu(dx)
 \\
&&= 
n\left(C_{\mu}\left(\frac{k-1}{n}\right)+C_{\mu}\left(\frac{k+1}{n}\right)-2C_{\mu}\left(\frac{k}{n}\right)\right),
\e*  
where $C_\mu(K)= \int_{\R}(x-K)^+\mu(dx)$ are the call prices encoded by $\mu$. We note that other discretisations, similar in spirit to \eqref{eq:DSdiscretisation} but distinct, are possible, see for example the $U$-quantisation in \cite{Baker:12}. 

The above discussion assumed we knew $\mu$ through its density or distribution 
function, or similar. If instead we are able to simulate i.i.d.\ random 
variables $(\xi_i)$ from $\mu$ then it is natural to approximate $\mu$ 
using the empirical measures $\hat{\mu}^n=\frac{1}{n}\sum_{k=1}^n 
\delta_{\xi_i}$ constructed from the samples. The distance 
$\Wc_1(\hat{\mu}^n,\mu)$ can be bounded relying on the results of 
\cite{fournier2015rate}, we refer to \cite{GO} for the details. We note that 
such approximations may not preserve the convex order. In light of Theorem 
\ref{thm:mot-lp}, this is not an issue for our methods but one may further 
consider $\Wc_1$-projections onto couples which are in convex order, see 
\cite{ACJ:19} for details. 

Finally, let us comment on the issue of convergence rates in Theorem \ref{thm:mot-lp}. For $d=1$ and $T=2$ such rates were obtained in \cite{GO} but, at present, remain open in greater generality. To obtain an estimation of the convergence rate, we 
    need not only to know the continuity of  $\mv\mapsto\overline{P}(\mv)$ -- 
    this has been settled for $d=1$ recently but remains open otherwise, see \cite{BVP:19,Wiesel:19} -- but also the differentiability (Lipschitz continuity) of 
    $\mv\mapsto\overline{P}(\mv)$, see \cite{GO}. Nevertheless, we hope 
    this may be achievable in the future and it is one of the reasons to 
    consider the LP approach.
    
\subsection{The Dual Problem - a Neural Network approach}
\label{sec:NNtheory}
We develop now a computational approach to the MMOT problem \eqref{def:mmot} 
based on a neural network implementation of the dual formulation 
\eqref{def:mmotdual}.
The basic idea, following the work of \cite{Eckstein2019} for the MOT problem, 
is to restrict $\varphi_{t,i}$, $h_{t,i}$ to neural network functions 
instead of arbitrary $\Lsp^1$ or $\Lsp^0$ functions. 
Without loss of generality, we restrict the discussion to the problem 
$\po=\ddo$.

Formally, we define 
\begin{align*}
\mathcal{H} := \Big\{ &h \in \Lsp^0(\Xc) : \exists (\varphi_{t, i}, h_{t, i}) 
\in \Do \text{ s.t.~for all } x \in \Xc\\
& h(x) = \sum_{t=1}^T \sum_{i=1}^d \varphi_{ t,i}(x_{t,i}) + \sum_{t=1}^{T-1} 
\sum_{i=1}^{d} h_{t,i }(x_1, ..., x_t) (x_{t+1, i} - x_{t, i}).
\Big\}
\end{align*}
Note that, for brevity, $h$ now denotes the combined payoff from dual elements $(\vp_{t, i},h_{t, i})$. For an arbitrary $\mu_0 \in {\mathcal M}(\mv)$ one can rewrite 
\[\ddo(\mv) = \inf_{h\in \mathcal{H}:\, h \geq c} \int h \,d\mu_0, \]
where the value $\ddo(\mv)$ clearly does not depend on the choice of $\mu_0$. We denote by $\mathfrak{N}_{l, k, m}$
the set of feed-forward neural network functions mapping $\mathbb{R}^k$ into 
$\mathbb{R}$, with $l$ layers and hidden dimension $m$. More precisely, we fix 
an activation function $\psi : \mathbb{R} \rightarrow \mathbb{R}$ and define
\begin{align*}
\mathfrak{N}_{l, k, m} = \{ f : \mathbb{R}^k \rightarrow \mathbb{R} : &\text{ 
There exist affine transformations } A_0, ..., A_l \text{ such that } \\ &f(x) 
= A_l \circ \psi \circ A_{l-1} \circ ... \circ \psi \circ A_0(x)\}
\end{align*}
whereby the index $m$ specifies that $A_0$ maps from $\mathbb{R}^k$ to 
$\mathbb{R}^m$, $A_1, ..., A_{l-1}$ map from $\mathbb{R}^m$ to $\mathbb{R}^m$ 
and $A_l$ maps from $\mathbb{R}^m$ to $\mathbb{R}$. The evaluation of $\psi(x)$ 
for $x \in \mathbb{R}^d$ (for some $d \in \mathbb{N}$) is understood point-wise, i.e.~$\psi(x) = (\psi(x_1), 
..., \psi(x_d))$.

Fix $l \in \mathbb{N}$ and define $\Do^m \subset \Do$ as the set of functions 
$(\varphi_{t, i}, h_{t, i})$ with $\varphi_{t, i} \in \mathfrak{N}_{l, 1, m}$ 
and $h_{t, i} \in \mathfrak{N}_{l, d \cdot t, m}$. Similarly, $\mathcal{H}^m 
\subseteq \mathcal{H}$ is defined by
\begin{align*}
\mathcal{H}^m := \Big\{ &h \in \Lsp^0(\Xc) : \exists (\varphi_{t, i}, h_{t, i}) 
\in \Do^m \text{ s.t.~for all } x \in \Xc\\
&h(x) = \sum_{t=1}^T \sum_{i=1}^d \varphi_{t,i}(x_{i, t}) + \sum_{t=1}^{T-1} 
\sum_{i=1}^{d} h_{t,i}(x_1, ..., x_t) (x_{t+1, i} - x_{t, i})\Big\}
\end{align*}
which leads to the problem 
\[
\ddo^m(\mv) := \inf_{h \in \mathcal{H}^m:\,h \geq c} \int h \,d\mu_0.
\]
Aside from the point-wise inequality constraint $h\geq c$, the problem 
$\ddo^m(\mv)$ fits into the standard framework of optimization problems 
for neural networks. This leads us to consider penalizing the inequality constraint. To do so, 
choose a penalty function $\beta: \mathbb{R} \rightarrow \mathbb{R}_+$ which is strictly  
increasing, convex and differentiable on $(0,\infty)$ with $\frac{\beta(x)}{x} 
\rightarrow \infty$ for $x\rightarrow \infty$. Define $\beta_\gamma : 
\mathbb{R} \rightarrow \mathbb{R}_+$ by $\beta_\gamma(x) := \frac{1}{\gamma} 
\beta(\gamma x)$. Further, choose a measure $\theta \in \mathcal{P}(\Xc)$. The 
penalized problem which can be solved numerically is given by
\[
\ddo^m_{\theta,\gamma}(\mv) := \inf_{h \in \mathcal{H}^m} \int h \,d\mu_0 + 
\int \beta_{\gamma}(c-h) \,d\theta.
\]
The penalization used for the Sinkhorn algorithm \cite{cuturi2013sinkhorn} corresponds to the choice $\beta(x) = \exp(x-1)$, while similar penalization methods for neural network based approaches usually utilize a power-type penalization, see also \cite{gulrajani2017improved,seguy2018large}. In our case, for instance, $\beta(x) = \max\{0, x\}^2$ will be used. It follows from Theorem \ref{thm:duality} and \cite[Lemma 3.3.~and Proposition 
3.7]{Eckstein2019} that this problem approximates $\ddo(\mv)$ in the following 
sense:
\begin{theorem}\label{thm:cnv_nn}
	Assume that $c$ is continuous and all marginals $\mu_{t, i}$ are compactly 
	supported:  $\mu_{t, i}([-M, M])=1$ for some $M > 0$ and all $1\leq t\leq 
	T$, $1\leq i\leq d$. For the neural networks, the activation function is continuous, nondecreasing, bounded and nonconstant, and there is at least one hidden layer.
	Consider $\ddo^m(\mv)$ as defined above but with the inequality constraint 
	restricted to $[-M, M]^{T\times d}$. Then
	\begin{align}
	\ddo^m(\mv) &\rightarrow \ddo(\mv) & \text{for } m \rightarrow \infty
	\end{align}
	and if the support of $\theta$ is equal to $[-M, M]^{T\times d}$ then also
	\begin{align} 
	\ddo^m_{\theta, \gamma}(\mv) &\rightarrow \ddo^m(\mv) & \text{for } \gamma 
	\rightarrow \infty.
	\end{align}
\end{theorem}

\begin{remark}
	The penalization of the inequality constraint has the added benefit that it 
	introduces a functional relation between dual and primal optimizers. Thus 
	in practice, one can easily obtain approximate primal optimizers from the 
	obtained neural network solutions. Formally, the problem 
	\[
	\ddo_{\theta,\gamma}(\mv) := \inf_{h \in \mathcal{H}} \int h \,d\mu_0 + 
	\int \beta_{\gamma}(c-h) \,d\theta
	\]
	has a primal problem of the form
	\[
	\po_{\theta, \gamma}(\mv) = \sup_{\pi\in \Mc(\mv)}~ \int c \,d\pi - \int 
	\beta_{\gamma}^*\Big(\frac{d\pi}{d\theta}\Big) \,d\theta.
	\]
	Here, $\beta_{\gamma}^*$ is the convex conjugate of $\beta_{\gamma}$ and 
	the Radon-Nikodym derivative $\frac{d\pi}{d\theta}$ is understood to be 
	infinite if $\pi$ is not absolutely continuous with respect to $\theta$.
	Then under the assumptions of Theorem \ref{thm:cnv_nn}, 
	any optimizer $\hat{h}_\gamma$ of $\ddo_{\theta,\gamma}(\mv)$ yields an optimizer 
	$\hat{\pi}_{\gamma}$ of $\po_{\theta, \gamma}(\mv)$ via
	\begin{align}
	\label{primal_formula}
	\frac{d\hat{\pi}_\gamma}{d\theta} = \beta_\gamma'(c - \hat{h}_\gamma),
	\end{align}
	see also \cite[Theorem 2.2]{Eckstein2019}. It further 
		holds
		\begin{align}
		\label{eq:ineqchain}
		\po_{\theta, \gamma}(\mv) \leq \int c
		\,d\hat{\pi}_\gamma - \beta_{\gamma}^\ast(1)\leq 
		\po(\mv) - \beta_{\gamma}^\ast(1)
		\end{align}
		and hence $\int c 
		\,d\hat{\pi}_\gamma$ 
		converges to
		$\po(\mv)$ for $\gamma \rightarrow \infty$ whenever $\lim_{\gamma 
			\rightarrow \infty}\po_{\theta, \gamma}(\mv) = \po(\mv)$ holds. The 
			latter 
		convergence, and particularly the correct conditions on $\theta$, is an 
		open problem even for MOT, see also  \cite[Theorem 5.5]{de2018entropic}. 
		Nevertheless, given that the convergence of values $\lim_{\gamma \rightarrow \infty}\po_{\theta, \gamma}(\mv) = \po(\mv)$ holds, \eqref{eq:ineqchain} above also implies that any limiting point of $(\pi_{\gamma})_{\gamma > 0}$ is an optimizer of $\po(\mv)$. Further, by tightness, one also knows that a convergent subsequence exists. Uniqueness of such a limit is however an open problem, not least since the optimizer of the $\po(\mv)$ does not need to be unique, as seen in Example \ref{Ex2}. 
\end{remark}

\subsection{The case of finitely many quoted call options}

So far we have assumed that market specified the risk-neutral distributions of 
each asset at the given maturities. Equivalently, we assumed that the set of 
traded strikes at these maturities was dense in $\R$. 
This allows us to use the language of measures and of optimal transportation 
but is a simplifying assumption: in practice only finitely many call options 
are liquidly traded. Observe that our numerical methods can easily address this 
point: in the NN method we simply restrict $\varphi_{t,i}$ in $\Do^m$ to linear 
combinations of the traded call options, see Section \ref{sec:FX} below. Likewise, in the LP implementation, we 
consider discrete measures supported on the traded strikes, in analogy to 
\eqref{eq:DSdiscretisation}. Moreover, we can establish convergence of the 
problems with finitely many constraints to the MMOT problem as the number of 
strikes increases. 

To this end fix $\mu\in \Pc(\R)$ with support bounds $-\infty\leq 
a_\mu<b_\mu\leq \infty$. 
Let $\bk^n:=\big\{a_\mu<K^n_1<\ldots<K^n_{m_n}<b_\mu\big\}$ be the set of strikes 
and $\bc^n:=\big\{C^n_i:=C_{\mu}(K^n_i): 1\le i\le m_n\big\}$ be the collection 
of the corresponding prices of call options. Naturally, we assume that this discrete set of strikes gets asymptotically dense in the following sense:
\begin{ass} \label{hyp:stab}
	As $n\to\infty$, one has
	\b*\label{eq:calls_agree}
	\Delta\bk^n:=\max_{2\le i\le 
	m_n}\left(K^n_i-K^n_{i-1}\right)\longrightarrow 0,\quad K_1^n\to a_\mu\quad 
	\mbox{and} \quad K^n_{m_n}\to b_\mu.
	\e*
\end{ass}  
The following result, together with Proposition \ref{prop:approx}, establishes 
sufficient conditions for the MMOT problems for measures $\mv^n$ matching only 
finitely many call prices from $\mv$ to converge to the MMOT problem for $\mv$. 
\begin{proposition}\label{prop:stab}
Let Assumption \ref{hyp:stab} hold. Then, for any sequence  $(\mu^n)_{n\ge 1}$ satisfying 
	\be\label{eq:finiteKagree}
	\int x 
	\mu^n(dx)=\int x \mu(dx)\equiv \lambda \mbox{ and }	C_{\mu^n}(K^n_i)~=~C_{\m}(K_i^n), \mbox{ for } i=1,\ldots, m_n,
	\ee
	we have 
	\b*
	\frac{1}{2}\Wc_1(\mu^n,\m) &\le& \Delta \bk^n+ K^n_1-\lambda+C_\m(K^n_1)+C_\m(K^n_{m_n})
	\e*
	and, in particular, $\Wc_1(\mu,\mu^n)\to 0$ as $n\to \infty$.
\end{proposition}

\begin{proof}
	Note that $K\to C_\mu(K)$ is $1$-Lipschitz continuous and decreasing with 
	$C_\mu(K)\to 0$ as $K\to b_\mu$ and $C_\mu(K)\ge \lambda-K$ with $C_\mu(K)+K\to \lambda$ as 
	$K\to a_\mu$. Fix $n\geq 1$. We claim that
	\be \label{eq:mineq}
	\left|C_{\m^n}(K)-C_{\mu}(K)\right|&\le& \Delta \bk^n,\quad \mbox{for all } K\in [K^n_1,K^n_{m_n}].
	\ee 
	Indeed, for each $K\in [K^n_{i},K^n_{i+1}]$ with some $i$, one has by definition 
	\b* 
    C_{\mu^n}(K) - C_{\mu}(K) \le C_{\mu^n}(K^n_i) - C_{\mu}(K^n_{i+1}) =  C_{\mu}(K^n_i) - C_{\mu}(K^n_{i+1}) \le \Delta \bk^n.
	\e*
	Similarly one has $C_{\mu}(K) - C_{\mu^n}(K)\le \Delta \bk^n$ and thus \eqref{eq:mineq} holds. 
	
	Consider first the particular case when 
	$\supp(\mu)\cup\supp(\mu^n)\subset [K^n_1,K^n_{m_n}]$. 
	Let $\nu^n$ be the measure supported inside $[K^n_1,K^n_{m_n}]$  with call prices defined via
	$$ C_{\nu^n}(K)=C_\mu(K)\lor C_{\mu^n}(K),\quad K\in \R.$$
	Note that $\mu \leqcvx \nu^n$, $\mu^n\leqcvx\nu^n$ and $C_{\nu_n}(K)=C_\mu(K)=C_{\mu^n}(K)$ for $K\in \bk^n$. 
	Consider a probability space $(\Omega, \mathbb{F}, \mathbb{P})$ supporting 
	a standard Brownian motion $(B_t)$. We can use any standard Skorokhod 
	embedding, e.g., the Chacon-Walsh embedding, see \cite{obloj2004skorokhod}, 
	to find stopping times $\tau\leq \rho$ such that $B_\tau \sim \mu$, 
	$B_\rho\sim \nu^n$ and $(B_{t\land \rho}:t\geq 0)$ is uniformly integrable. 
	The latter property implies in particular that if $K\in [K^n_i,K^n_{i+1}]$ 
	then, conditionally on $\{B_\tau =K\}$, we have $B_\rho \in 
	[K^n_{i},K^n_{i+1}]$. Put differently, we have $|B_\tau-B_\rho|\leq \Delta 
	\bk^n$ and, in particular, $\Wc_1(\mu,\nu^n)\leq \Delta\bk^n$. Likewise, we 
	obtain $\Wc_1(\mu^n,\nu^n)\leq \Delta\bk^n$ and, in conclusion, 
	$\Wc_1(\mu,\mu^n)\leq 2 \Delta \bk^n$.

For the case of general supports we introduce auxiliary measures. Let $Z$ and 
$Z^n$ be random variables distributed according to $\mu$ and $\m^n$ 
respectively. 
	Denote by $\tilde \mu$ and $\tilde \m^n$  the laws of $\tilde Z:=K^n_{m_n}\wedge (K^n_1 \vee Z)$ and 
	$\tilde Z^n:=K^n_{m_n}\wedge (K^n_1 \vee Z^n)$. Note that, for $K\in [K^n_1,K^n_{m_n}]$, 
	$$ C_{\tilde \m}(K)=\int_{K}^\infty (x\wedge 
	K^n_{m_n}-K)\mu(dx)=C_\mu(K)-C_\mu(K^n_{m_n}),$$
	with an analogue expression for $C_{\tilde \m^n}$. Further, $C_{\tilde \mu}(K)=C_{\tilde \mu_n}(K)$ for all $K\notin [K_1^n,K^n_{m_1}]$. In particular,
	$$C_{\tilde \mu}(K_1^n)=\E[\tilde Z]=C_{\mu}(K^n_1)-C_{\mu}(K^n_{m_n})=C_{\mu^n}(K^n_1)-C_{\mu^n}(K^n_{m_n})=\E[\tilde Z^n].$$
	It follows that \eqref{eq:finiteKagree} hold for $\tilde \mu$ and $\tilde \mu^n$ and, by the above, $\Wc_1(\tilde \mu, \tilde \mu^n)\leq 2\Delta \bk^n$. 
	Finally,
	$$\Wc_1(\tilde \mu,\mu) \leq \E[(K_1^n-Z)^+]+\E[(Z-K^n_{m_n})^+] = 
	K_1^n-\E[Z]+C_\m(K_1^n) + C_\m(K^n_{m_n}),$$
	with the same bound valid for $\Wc_1(\tilde \mu^n,\mu^n)$ by 
	\eqref{eq:finiteKagree}. The result follows by the triangular inequality. 
\end{proof}

\section{Numerical Examples}\label{sec:num_ex}
We turn now to numerical results. We implement both methodologies presented 
above: the LP approach of Section \ref{sec:LPtheory} and the NN approach of 
Section \ref{sec:NNtheory}. Our first aim is to showcase that 
both methods are reliable. This is achieved via a comprehensive testing of their performance on a range of examples. In the process, we also discuss the  respective advantages and drawbacks of the two methods. Our second aim is to illustrate the capacity of the MMOT approach to capture and quantify, in a fully non-parametric way, the influence of market inputs on a given pricing and hedging problem. This is achieved by showing how adding additional information sharpens the bounds by reducing $\po-\pu$, the relative range of no arbitrage prices.\footnote{Python code to reproduce the examples, based on TensorFlow for the neural network implementation and Gurobi for the linear programs, can be found at \url{https://github.com/stephaneckstein/superhedging/tree/master/Examples/MMOT} .}

Throughout the examples we mostly work with $d=2$ but also consider $d=3$. We are interested in comparing 
results when we vary the number of maturities, or time points, $T$. To enable such a 
comparison, we mostly consider cost functions that only depend on the final 
time point. More precisely, we focus mostly on:
\begin{align*}
c(x) &:= |x_{T, 1} - x_{T, 2}|^p ~ &\text{(spread option)} \\
c(x) &:= (x_{T, 1} + x_{T, 2} - K)^+. ~ &\text{(basket option)}
\end{align*}
We first assume knowledge of only the marginal distributions at the final time 
point and compute the highest and lowest possible prices for a cost function 
$c$ under these marginal constraints. These correspond to the optimal transport 
bounds $\ou,\oo$ in \eqref{def:ot}. Then we additionally assume that marginals 
at earlier time steps are known. The knowledge of marginal distributions at 
earlier time steps, combined with the martingale condition, further constrains   
the possible joint distributions at the final time point. We can then study the 
degree to which this narrows the price bounds.

\subsection{Uniform marginals}
\label{subsec:Num1}
\def\arraystretch{1.25}

\begin{table}
	\begin{minipage}{0.5\textwidth}
		\begin{tabular}{l P{0.8cm} P{0.8cm} P{0.8cm} P{0.8cm}}
			\noalign{\global\arrayrulewidth=0.1mm}\hline\noalign{\global\arrayrulewidth=0.1mm}
			& \multicolumn{4}{p{3.2cm}}{\centering Spread Option} \\
			$t$ & 1 & 2 & 3 & 4 \\ 
			\noalign{\global\arrayrulewidth=0.4mm}\hline\noalign{\global\arrayrulewidth=0.1mm}
			 \vspace{1mm}
			$x_{t, 1}$ & 1 & 1.6 & 2.5 & 3  \\
			$x_{t, 2}$ & 1 & 1.5 & 1.6 & 2  \\ \hline
		\end{tabular}
	\end{minipage}\hfill
	\begin{minipage}{0.5\textwidth}
		\begin{tabular}{l P{0.8cm} P{0.8cm} P{0.8cm} P{0.8cm}}
			\noalign{\global\arrayrulewidth=0.1mm}\hline\noalign{\global\arrayrulewidth=0.1mm}
			& \multicolumn{4}{p{3.2cm}}{\centering Basket Option} \\
			$t$ & 1 & 2 & 3 & 4 \\ 
			\noalign{\global\arrayrulewidth=0.4mm}\hline\noalign{\global\arrayrulewidth=0.1mm}
			 \vspace{1mm}
			$x_{t, 1}$ & 1 & 1.75 & 2 & 3  \\
			$x_{t, 2}$ & 2 & 2.1 & 2.3 & 3  \\ \hline 
		\end{tabular}
	\end{minipage}\vspace{1mm}
	\caption{Details for the marginal distributions in Section 
	\ref{subsec:Num1}. 
		Each marginal distribution $\mu_{t, i}$ is uniform on 
		the interval $[-x_{t, i}, x_{t, i}]$. In the examples, in case $T=1$, 
		only the information at $t=4$ is used. If $T=2$, the time steps $t = 1, 
		4$ are used. And for $T=4$, all time steps are included.}
	\label{table:marginals}
\end{table}
\def\arraystretch{1}
We first consider a simple example where all occurring marginal distributions 
are uniform, see Table \ref{table:marginals}. For both spread and basket 
option, Table \ref{table:comparisonLPNN} compares the two numerical approaches 
introduced in Sections \ref{sec:LPtheory} and \ref{sec:NNtheory}. For the 
linear programming method, we discretize as shown in Appendix 
\ref{app:discrete}. For the neural network implementation, we use the 
network 
architecture described in \cite[Section 4]{Eckstein2019}.

First, in Table \ref{table:comparisonLPNN}, we consider marginal distribution 
constraints at two maturities and then, in Table \ref{table:improvedbounds}, 
extend it to four maturities. For the latter, only the numerical values obtained by the neural 
network implementation are reported, as the discretized LP problem is too large 
to solve in the case of four time steps\footnote{We note that there are heuristics which allow LP methods to tackle such problems, a prime example being the cutting plane method used by \cite{phlautomated}. However, these correspond to a significantly different approach than pursued in Section \ref{sec:LPtheory} and introduce qualitatively new types of errors. We do not employ these methods as it would make a comprehensive study of numerics even harder.}.
Finally, Figures \ref{fg:OptimizerSpread} and \ref{fg:OptimizerBasket} show how 
the numerically optimal couplings between the two assets at the final time 
point change with the inclusion of more information from previous time steps.

\def\arraystretch{1.25}
\begin{table}
	\caption{Comparison of the optimal values obtained using different numerical 
	approaches}
	\begin{tabular}{ p{0.55cm} P{1.025cm} P{1.025cm}  P{1.025cm} P{1.025cm}  
	P{1.025cm} 
	P{1.025cm}  P{1.025cm} P{1.025cm}  }
		\noalign{\global\arrayrulewidth=0.1mm}\hline\noalign{\global\arrayrulewidth=0.1mm}
		& \multicolumn{2}{p{2.05cm}}{~~\centering$\mmoto$} & 
		\multicolumn{2}{p{2.05cm}}{~\centering$\oo$} & 
		\multicolumn{2}{p{2.05cm}}{~~\centering$\mmotu$} & 
		\multicolumn{2}{p{2.05cm}}{~\centering$\ou$} \\
		& {\centering LP} & {\centering NN} & {\centering LP} & {\centering NN} 
		& {\centering LP} & {\centering NN} & {\centering LP} & {\centering 
		NN}  \\ 			
		\noalign{\global\arrayrulewidth=0.4mm}\hline\noalign{\global\arrayrulewidth=0.1mm}
		 \vspace{1mm}
		$p$ & \multicolumn{8}{c}{Spread Option}\\
		$1/2$ & $1.578$ & 1.577 & $1.578$ & 1.577 &  $0.383$ & 0.396 & $0.383$ 
		& 0.391  \\
		$1$  & $2.500$ & 2.500 & $2.500$ & 2.500 & $0.500$ & 0.501 & $0.500$ & 
		0.500 \\
		$2$  & $8.273$ & 8.254 & $8.338$ & 8.337 & $0.401$ & 0.416 & $0.335$ & 
		0.335  \\
		$3$ & $31.16$ & 31.24 & $31.29$ & 31.25 & $0.301$ & 0.321 & $0.253$ & 
		0.253  \\\vspace{1mm}
		$K$ & \multicolumn{8}{c}{Basket Option}\\
		$-1$ & $2.042$ & 2.041 & $2.042$ & 2.041  &  $1.000$& 1.000 & $1.000$ & 
		1.000  \\
		$0$ & $1.500$ & 1.500 & $1.500$ & 1.500 & $0.250$ & 0.260 & $0.000$ & 
		0.025 \\
		$1$ & $1.042$ & 1.041 & $1.042$ & 1.041 & $0.000$ & 0.006 & $0.000$ & 
		0.000   \\
		$2$ & $0.667$ & 0.667 & $0.667$ & 0.667 & $0.000$ & 0.000 & $0.000$ & 
		0.000  \\
		\hline
	\end{tabular}\vspace{1mm}
	\caption*{Optimal values for the example in Section \ref{subsec:Num1} example and the case $T=2$. 
	For the linear programming (LP) method, marginals are discretized in convex 
	order using the method in Appendix \ref{app:discrete}. The penalty function 
	for the neural network implementation is $\beta_\gamma(x) = \gamma \cdot 
	x_+^2$ where $\gamma$ is set to $2500$ times the number of time steps in 
	the optimization problem.}\label{table:comparisonLPNN}
\end{table}
\def\arraystretch{1}

In Table \ref{table:comparisonLPNN} we observe that in the simple examples 
considered, the two numerical approaches agree in most of the cases. In some 
cases, like for the spread option $(p=2)$ and the problem $\po$, there are 
slight differences between the optimal value obtained by the neural network 
implementation (8.254) and the linear programming approach (8.273). For the 
neural network implementation, we believe the biggest source of numerical error 
arises from the penalization of the inequality constraint in the dual 
formulation. Since the penalization decreases the upper bound (i.e., 
$\Do^m_{\theta, \gamma} \leq \Do^m$, see \cite[Theorem 2.2]{Eckstein2019} and 
note that for the quadratic penalization used here, it holds $\beta(0) = 0$) and increases the lower bound, the reported bounds by the neural network method 
are likely slightly more narrow than the true analytical bounds. By choosing 
$\gamma$ large enough this effect can be minimized.\footnote{By doing so, one 
must consider the numerical stability of the resulting problem. If $\gamma$ is 
too large, gradients explode and the numerical optimization procedure will not 
find the true optimizer, which leads to a different kind of numerical error. 
For the problems considered, $\gamma$ was gradually increased (while 
simultaneously increasing the batch size in the numerical implementation for 
stability) so that no further change in optimal values could be observed.} For 
the linear programming method, one cannot make a similar estimation for whether 
the obtained numerical bounds are narrower or wider than the true bounds. The 
main (and in this example only) approximation error for the linear programming 
implementation arises from discretization, which in priniciple can both increase or decrease 
optimal values. We comment further on the monotonicity of the approximations below. 

By observing values by both the LP and NN method, one can obtain greater trust 
in 
the obtained values, whenever they coincide. The reason is that both 
computational methods have entirely different sources of error, and hence 
whenever the obtained numerical values (almost) agree, it suggests that both 
errors are in fact small, since it is unlikely that the different error sources 
should produce the same incorrect value instead. Nevertheless, in cases where 
the two methods do not agree, the question arises how to obtain more certainty 
about the true value. For the LP method, simply observing the evolution of 
values in the parameter $n$ can give a clearer picture. The same is true for 
the NN 
method with the parameter $\gamma$. Further, since the NN method is based on a 
stochastic algorithm, running the optimization several times can give better 
indicators of the true value. We performed such an analysis in Figure 
\ref{fg:convergence_analysis} for the aforementioned case of the spread option 
($p = 2$). We observe two patterns in Figure \ref{fg:convergence_analysis}. First, $n \mapsto \mmoto(\mv^n)$ is decreasing, and $\gamma \mapsto \mmoto_{\theta, \gamma}(\mv)$ is increasing. The latter is easily derived from the definition, as mentioned above. For the mapping $n \mapsto \mmoto(\mv^n)$, two effects are at work. First, the marginal distributions simply change and so the optimal value also changes. Intuitively, this first effect can be seen as the one which determines the nature of the mapping $n \mapsto \oo(\mv^n)$, and it has no monotonicity. The second effect is the effect of the martingale constraint. The (relaxed) martingale constraint is a constraint involving each element of the support and it becomes more restrictive as more support points are added. Further, the more points we add the closer we approximate the target marginals and hence the less slack we allow from the martingale property. Together, these effects, in our experience, dominate and explain the decreasing nature of the mapping $n\mapsto \mmoto(\mv^n)$.

\begin{figure}
	\caption{Numerical convergence analysis for the case of a spread option and 
	$p=2$ from Table \ref{table:comparisonLPNN}}\vspace{4mm}
	\centering
	\vspace*{-5mm}
	\begin{minipage}[b]{0.75\textwidth}
		\includegraphics[width=\textwidth,height=0.7\textwidth]{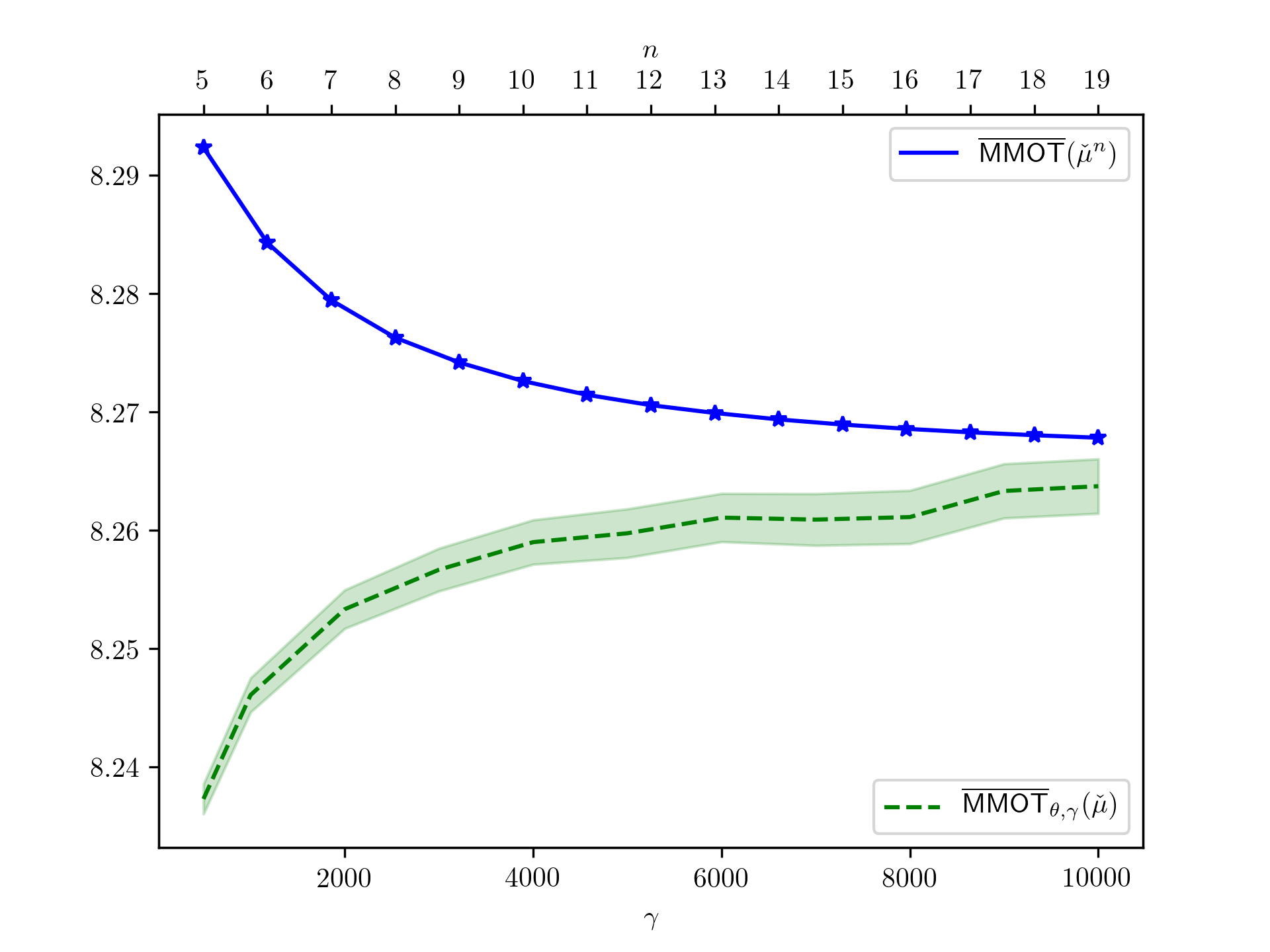}
	\end{minipage}\hspace{-5mm}
	\hfill
	\caption*{The blue line shows how the numerical values obtained by the LP 
	method depend on the discretisation parameter $n$, see Appendix 
	\ref{app:discrete}. The green line shows the analogue for 
	the NN method, and the penalization parameter $\gamma$. Since the numerical 
	method in the NN case is based on stochastic gradient descent, the final 
	values can vary. The green error bands indicate the standard deviation of 
	the obtained numerical values across 10 different sample runs. The higher 
	the penalization parameter $\gamma$ is chosen, the more the final values vary.}
	\label{fg:convergence_analysis}
\end{figure}

\def\arraystretch{1.25}
\begin{table}
	\caption{Improvement of bounds with information from additional maturities}
	\begin{tabular}{ l P{1.15cm} P{1.15cm} P{1.15cm} P{1.15cm} P{1.15cm} 
	P{1.15cm}  }
		\noalign{\global\arrayrulewidth=0.1mm}\hline\noalign{\global\arrayrulewidth=0.1mm}
		& {\centering$\oo$} & \multicolumn{2}{p{2.3cm}}{\centering~~$\mmoto$} & 
		{\centering$\ou$} & \multicolumn{2}{p{2.3cm}}{\centering~~$\mmotu$} \\
		$T$ & {\centering 1} & {\centering 2} & {\centering 4} & {\centering 1} 
		& {\centering 2} & {\centering 4}  \tabularnewline 
		\noalign{\global\arrayrulewidth=0.4mm}\hline\noalign{\global\arrayrulewidth=0.1mm}
		 \vspace{1mm}
		\begin{tabular}{l}Spread Option \\ ($p=2$) \end{tabular}  & 8.337 & 
		8.254  & 7.920 & 0.335 & 0.416 & 0.776 \\
		\begin{tabular}{l}Basket Option \\ ($K=0$) \end{tabular} & $1.500$ & 
		1.500 & 1.501 & 0.025 & 0.260 & 0.345 \\
		\hline
	\end{tabular}\vspace{1mm}\\
	\caption*{Numerically optimal values for the example in Section \ref{subsec:Num1} obtained by the NN implementation. 
	The penalization uses $\beta_\gamma(x) = 
	\gamma \cdot x_+^2$ where $\gamma$ is set to $2500$ times the number of 
	time steps in the optimization problem.}
	\label{table:improvedbounds}
\end{table}
\def\arraystretch{1}
\begin{figure}
	\caption{Spread Option ($p=2$). Numerically optimal couplings at the final 
	time point obtained using the NN approach.}\vspace{4mm}
	\begin{minipage}[b]{0.35\textwidth}
		\centering{Maximizer, $T=1$}
		\includegraphics[width=\textwidth,height=0.7\textwidth]{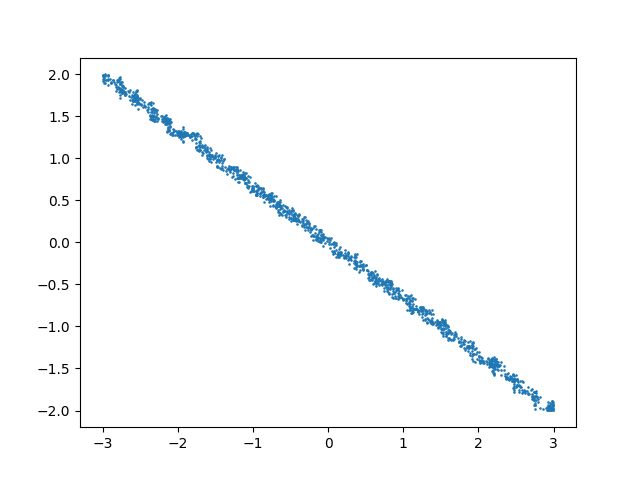}
	\end{minipage}\hspace{-5mm}
	\hfill
	\begin{minipage}[b]{0.35\textwidth} 
		\centering{Maximizer, $T=2$}
		\includegraphics[width=\textwidth,height=0.7\textwidth]{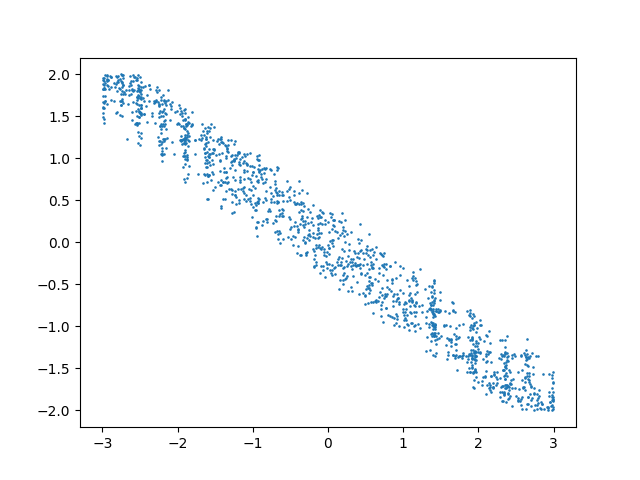}
	\end{minipage}\hspace{-5mm}
	\hfill
	\begin{minipage}[b]{0.35\textwidth}  
		\centering{Maximizer, $T=4$}
		\includegraphics[width=\textwidth,height=0.7\textwidth]{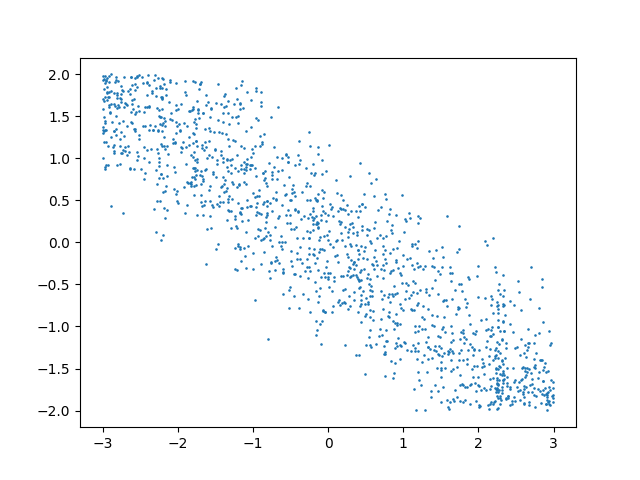}
	\end{minipage}
	
	\vspace{5mm}
	\begin{minipage}[b]{0.35\textwidth}
		\centering{Minimizer, $T=1$}
		\includegraphics[width=\textwidth,height=0.7\textwidth]{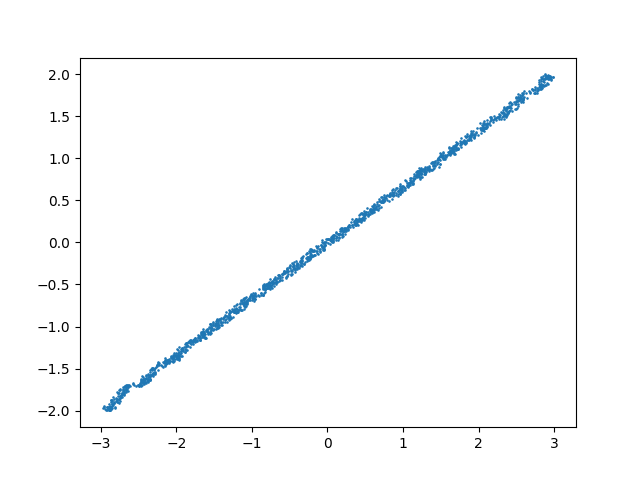}
	\end{minipage}\hspace{-5mm}
	\hfill
	\begin{minipage}[b]{0.35\textwidth} 
		\centering{Minimizer, $T=2$}
		\includegraphics[width=\textwidth,height=0.7\textwidth]{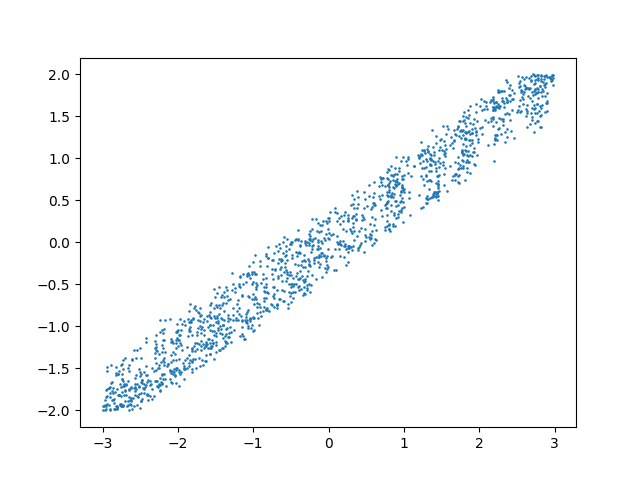}
	\end{minipage}\hspace{-5mm}
	\hfill
	\begin{minipage}[b]{0.35\textwidth}  
		\centering{Minimizer, $T=4$}
		\includegraphics[width=\textwidth,height=0.7\textwidth]{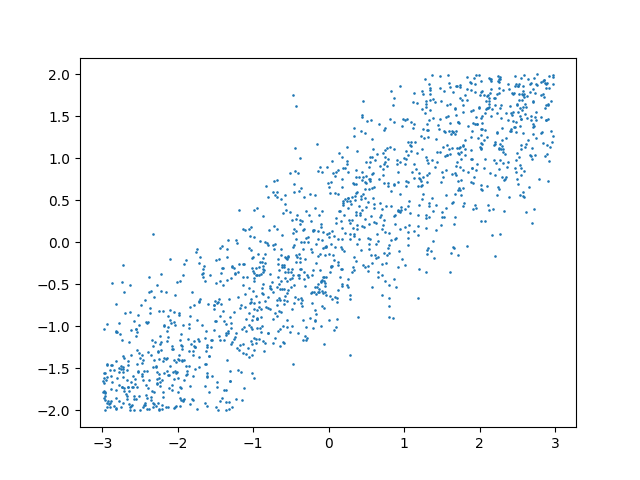}
	\end{minipage}
	\hfill
	\label{fg:OptimizerSpread}
\end{figure}

\begin{figure}
	\caption{Basket Option ($K=0$). Numerically optimal couplings at final time 
	point obtained using the NN approach.}\vspace{4mm}
	\begin{minipage}[b]{0.35\textwidth}
		\centering{Maximizer, $T=1$}
		\includegraphics[width=\textwidth,height=0.7\textwidth]{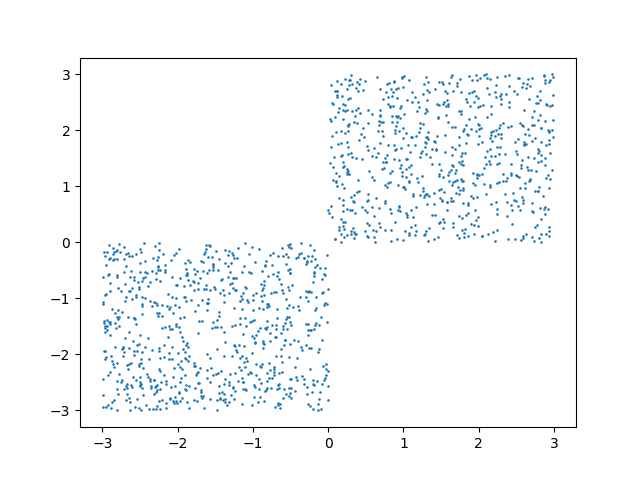}
	\end{minipage}\hspace{-5mm}
	\hfill
	\begin{minipage}[b]{0.35\textwidth} 
		\centering{Maximizer, $T=2$}
		\includegraphics[width=\textwidth,height=0.7\textwidth]{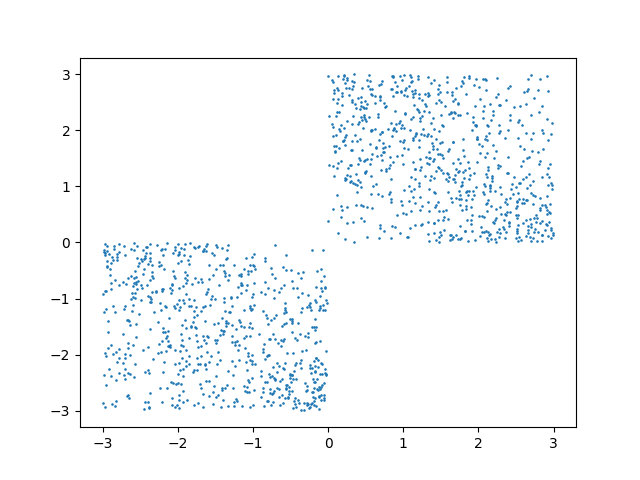}
	\end{minipage}\hspace{-5mm}
	\hfill
	\begin{minipage}[b]{0.35\textwidth}  
		\centering{Maximizer, $T=4$}
		\includegraphics[width=\textwidth,height=0.7\textwidth]{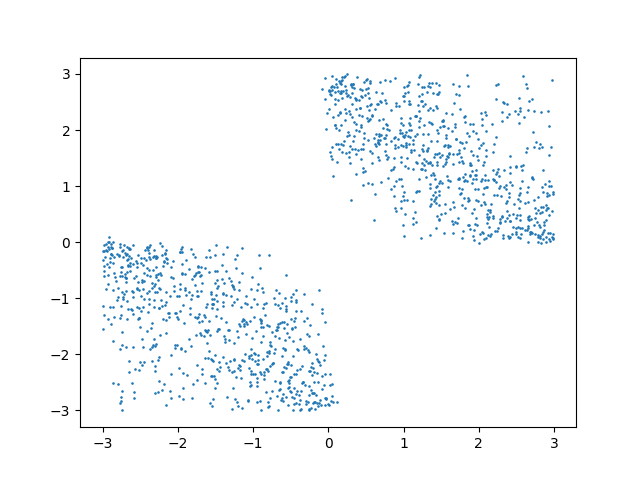}
	\end{minipage}
	
	\vspace{5mm}
	\begin{minipage}[b]{0.35\textwidth}
		\centering{Minimizer, $T=1$}
		\includegraphics[width=\textwidth,height=0.7\textwidth]{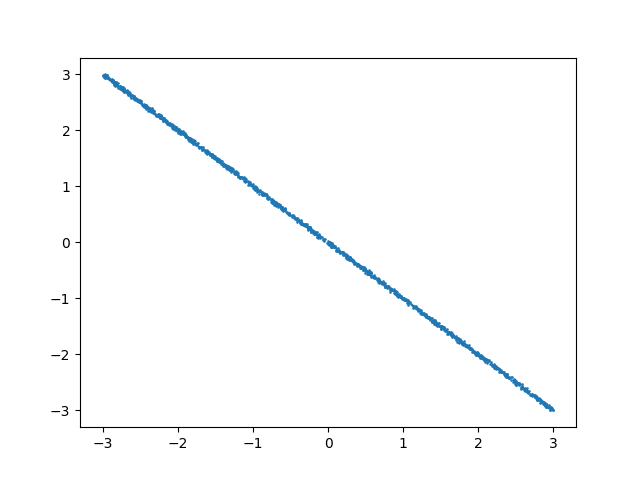}
	\end{minipage}\hspace{-5mm}
	\hfill
	\begin{minipage}[b]{0.35\textwidth} 
		\centering{Minimizer, $T=2$}
		\includegraphics[width=\textwidth,height=0.7\textwidth]{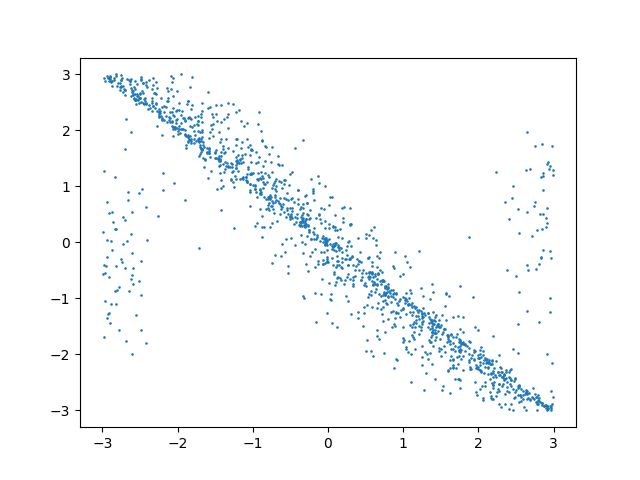}
	\end{minipage}\hspace{-5mm}
	\hfill
	\begin{minipage}[b]{0.35\textwidth}  
		\centering{Minimizer, $T=4$}
		\includegraphics[width=\textwidth,height=0.7\textwidth]{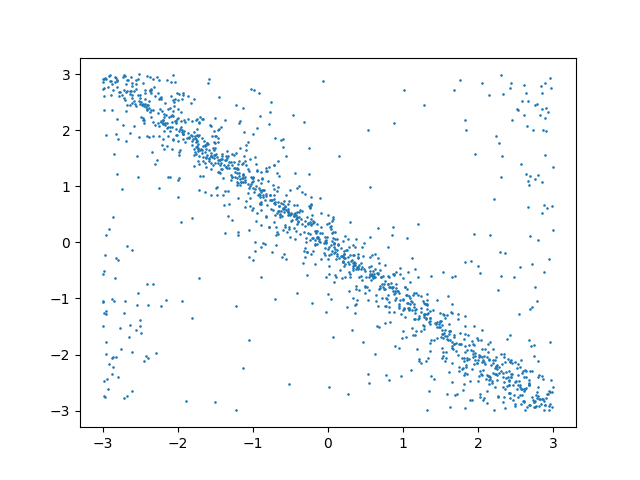}
	\end{minipage}
	\hfill
	\label{fg:OptimizerBasket}
\end{figure}

\begin{figure}
	\caption{Numerically optimal couplings at the final time point using the LP 
	approach for $T=2$.}\vspace{4mm}
	
	\vspace{1mm}
	\begin{minipage}[b]{0.5\textwidth}
		\centering{Maximizer Spread ($p=2$)}
		\includegraphics[width=\textwidth,height=0.7\textwidth]{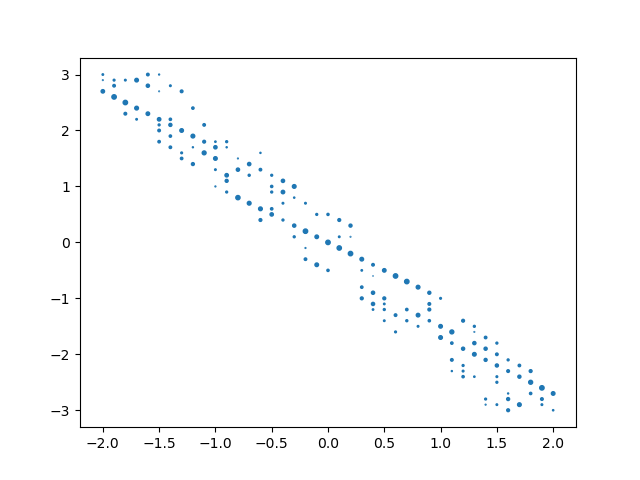}
	\end{minipage}\hspace{-5mm}
	\hfill
	\begin{minipage}[b]{0.5\textwidth} 
		\centering{Minimizer Spread ($p=2$)}
		\includegraphics[width=\textwidth,height=0.7\textwidth]{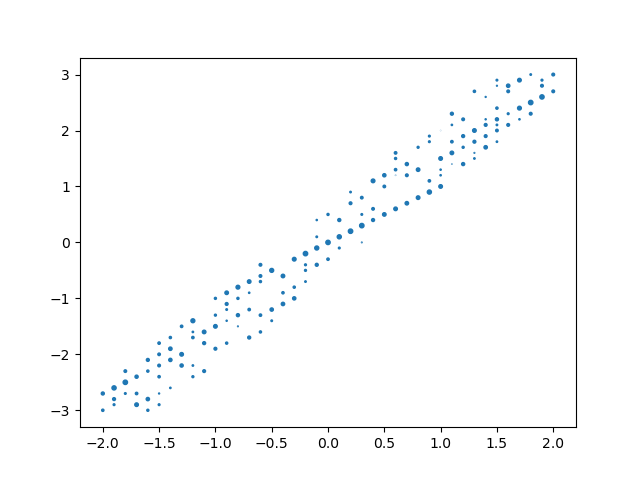}
	\end{minipage}\hspace{-5mm}
	\hfill
	
	\begin{minipage}[b]{0.5\textwidth}
		\centering{Maximizer Basket ($K=0$)}
		\includegraphics[width=\textwidth,height=0.7\textwidth]{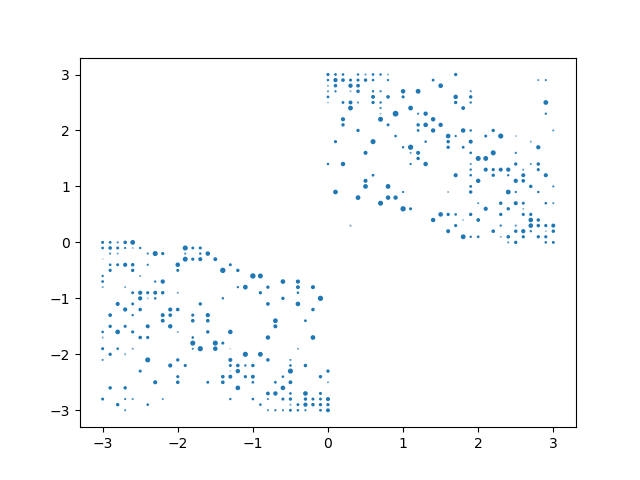}
	\end{minipage}\hspace{-5mm}
	\hfill
	\begin{minipage}[b]{0.5\textwidth} 
		\centering{Minimizer Basket ($K=0$)}
		\includegraphics[width=\textwidth,height=0.7\textwidth]{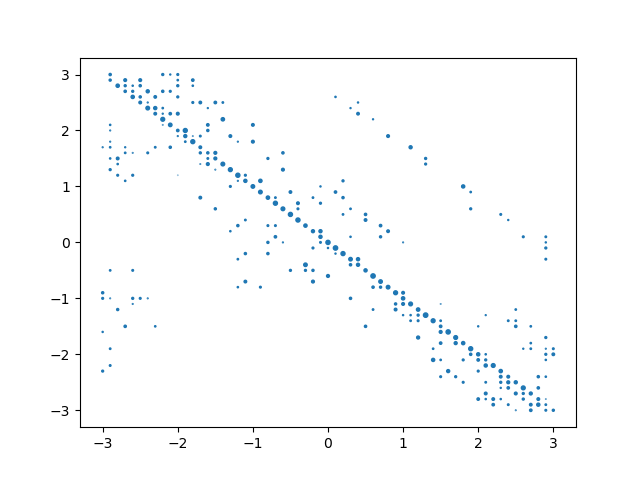}
	\end{minipage}\hspace{-5mm}
	\hfill
	
	\label{fg:LP_Optimizer}
\end{figure}

Having built confidence in the numerical precision of our methods, we now turn to examining how they capture and price the effect of using additional information. Recall that Table \ref{table:marginals} summarises the marginal distribution information available in the context of the simple example studied in this section. Table \ref{table:improvedbounds} shows the difference in numerical bounds from 
working with marginal distributions at 1, 2, or 4 maturities. We see that for both the 
spread and the basket option, significantly narrower bounds are obtained with each 
additional piece of information. The absolute bounds are still 
quite wide even with four time steps of information used: $(0.78, 7.92)$ for 
the spread and $(0.35, 1.50)$ for the basket option. This suggests that 
applicability of the obtained bounds as a pricing tool will be case-dependent. 
However, in all cases, it is the relative comparison of how the bounds behave 
across assets and when additional information is added which is informative. It 
gives quantitative insight into dependence and structural implications of 
pricing information across assets and maturities. To narrow bounds further we 
would need to include modelling assumption or significantly constraining new 
information, cf.~\cite{phlautomated,lutkebohmert2018tightening}.
We note that in the case of the upper bound for the 
basket option the additional information did not change the bound indicating the additional information is not relevant for this upper no-arbitrage price. We believe it is a strength of the methodology we present here to be able to pick up also such cases. In this particular case, the reasons can be understood analytically. Indeed, let us focus on the case $K=0$ in Table \ref{table:improvedbounds} but similar comments apply to other strikes in Table \ref{table:comparisonLPNN}, as well as to results presented in Table \ref{table:lognormal} in the next section. Using $(x+y)^+\leq x^++y^+$ we have
$$ \mmoto(\mv) = \sup_{\pi\in \Mc(\mv)}~ \int  (x_{T, 1} + x_{T, 2})^+ \,d\pi\leq \int x_{T,1} d\mu_{T,1}+\int x_{T,2} d\mu_{T,2}
$$
and this upper bound is independent of $\pi$ and is attained by any $\pi\in \Mc(\mv)$ for which $x_{T,1}$ and $x_{T,2}$ have the same sign $\pi$-a.s. This is a weak requirement and is typically attained by many couplings\footnote{Nevertheless we may come up with marginals for which this is not true. It we consider $T=2$ and marginals as in Table \ref{table:marginals} but we change $\mu_{1,2}$ to be uniform on $[-3,3]$ this forces the second asset to be constant through time and decreases the upper bound from $1.5$ to $1.3799$.}, as seen in Figure \ref{fg:OptimizerBasket} below. 

Figures \ref{fg:OptimizerSpread} and \ref{fg:OptimizerBasket} showcase the 
joint distribution between the first asset ($x$-axis) and the second asset 
($y$-axis) at the final time point. These are obtained using the NN approach via  
\eqref{primal_formula}. 
As expected, for the cases $T=2$ the depicted 
optimizers look very similar to the ones obtained by linear programming displayed in Figure \ref{fg:LP_Optimizer}.  The most notable characteristic of the 
observed optimizers is that in most cases (again, except for the supremum 
problem of the basket option), the optimal couplings become smoother when more 
time steps are involved. This is an interesting feature: where the OT problem 
returns a deterministic (Monge) coupling, when we add the martingale constraint 
the Monge coupling is not feasible but the optimizers are still concentrated on lower 
dimensional sets, see \cite{ghoussoub2019structure}. When we add further 
time points it adds more constraints and the models become less and less 
singular, i.e., having a more diffused support. 

\subsection{Lognormal marginals}
\label{subsec:lognormal}
We turn now to distributions more representative of real market conditions. Specifically, instead of uniform marginals we consider lognormal ones. 
As before, we consider two assets and, in this case, three distinct maturities. The way the 
marginal distributions are set up is that both assets have the same 
distributions at time points $1$ and $3$, but at time $2$, the marginals vary. 
In particular, the marginal distributions imply that the first asset accumulates most of its volatility between time points 
$2$ and $3$, while the second asset accumulates most if its volatility between time 
points $1$ and $2$. More precisely, we set $\mu_{t, i} \sim \exp(\sigma_{t, i} X - 
\sigma_{t, i}^2 
/2)$, where $X$ follows a standard normal distribution and $\sigma_{1, 1} = 
\sigma_{1, 2} = 0.1, \sigma_{3, 1} = \sigma_{3, 2} = 0.2, \sigma_{2, 1} = 0.11, 
\sigma_{2, 2} = 0.19$. 

We first calculate price bounds using only 
marginal information and trading between the first and third time points. Then, we include 
the intermediate maturity (the second time point) as well. This brings the additional marginal information which implies the asymmetry in the way the two assets accumulate their volatility as well as the ability to re-balance the hedging position at the intermediate time point. 
For the implementation, the neural network method remains unchanged compared to 
Section \ref{subsec:Num1}, just larger batch size is used to cope with the 
added difficulty of unbounded support of 
the marginals. For the LP method, we now use discretisation as described in 
\cite[Equation (6.5)]{ACJ:19}.\footnote{To be precise, for the 
	cases $T=2, 3$, the neural network implementation uses batch size $2^{13}, 
	2^{15}$, while the LP method uses $n = 39, 11$ support points for each marginal.}
\def\arraystretch{1.25}
\begin{table}
	\centering
	\caption{Improvement of bounds for the lognormal marginals}
	\begin{tabular}{P{1.025cm} P{1.025cm}  P{1.025cm} P{1.025cm}  
			P{1.025cm} 
			P{1.025cm}  P{1.025cm} P{1.025cm}  }
		\noalign{\global\arrayrulewidth=0.1mm}\hline\noalign{\global\arrayrulewidth=0.1mm}
		\multicolumn{2}{p{2.1cm}}{\centering $\mmoto$ \\ $T=2$} & 
		\multicolumn{2}{p{2.1cm}}{\centering $\mmoto$ \\ $T=3$} & 
		\multicolumn{2}{p{2.1cm}}{\centering $\mmotu$ \\ $T=2$} & 
		\multicolumn{2}{p{2.1cm}}{\centering $\mmotu$ \\ $T = 3$} \\
		{\centering LP} & {\centering NN} & {\centering LP} & {\centering NN} 
		& {\centering LP} & {\centering NN} & {\centering LP} & {\centering 
			NN}  \\ \noalign{\vskip 2mm}
		\multicolumn{8}{c}{Spread Option ($p = 2$)}\\
		0.1587 & 0.1625 & 0.1320 & 0.1359 &  0.0000 & 0.0010 & 0.0244
		& 0.0269 \\ \noalign{\vskip 2mm}
		\multicolumn{8}{c}{Basket Option (at the money, $K=2$)}\\
		0.1593 & 0.1593 & 0.1585 & 0.1593 & 0.0192 & 0.0191 & 0.0397 & 
		0.0423 \\
		\hline
	\end{tabular}\vspace{1mm}
	\caption*{Numerically optimal values for Section \ref{subsec:lognormal} lognormal example obtained using the linear programming (LP) and neural network (NN) implementations. The case $T = 2$ uses only the marginal information at the first and the last time point, while for the case $T=3$ the marginal distributions at an intermediate maturity are also given along with the ability to rebalance the hedging positions.}\label{table:lognormal}
\end{table}
\def\arraystretch{1}

Table \ref{table:lognormal} reports the resulting values. We see that the bounds tighten significantly with the addition of the intermediate maturity information and hedging, highlighting the capacity of our methods to capture and quantify the benefit of such additional information for pricing problems. 
The only example is given by the upper bound for the basket option, which was expected as explained in Section \ref{subsec:Num1}. 
Despite the added difficulty of the non-compactly supported marginals, as compared to Section \ref{subsec:Num1}, the LP and NN methods still produce very similar values in all cases. Most importantly, the effect of the improved bounds by including the additional time point is clearly more significant than the differences between the numerical values, and hence the qualitative message one can derive from this example is robust with respect to the numerical method used.

\subsection{A test of accuracy: comparison to theoretical values}
\label{subsec:accuracy}
In this section we consider a sanity check example where we can compare the numerical values to theoretical ones. 
Since generally, it is very difficult to obtain theoretical values for MMOT problems, we must refer to a structurally 
simple problem. To this end, consider uniform marginals: $\mu_{1, i} = \mathcal{U}([-1, 1])$ and 
$\mu_{2, i} = \mathcal{U}([-2, 2])$ for $i=1, ..., d$, and a cost function 
$c(x) := \sum_{i=1}^d \sum_{j=1}^d c_{i, j} x_{2, i} x_{2, j}$ for some $c_{i, j} \geq 0$ chosen randomly in the interval $[0, 1]$. Such costs are studied further in Section \ref{sec:structure}. An optimizer for this problem is 
characterized by the comonotone coupling among the dimensions at the final time 
step. The associated value can thus be computed to an arbitrary 
precision by sampling. 
Nevertheless, this analytical simplicity does not imply that the example is particularly easy to solve numerically. The 
singular nature of the optimizers is a feature which presents a challenge for numerical convergence.

Table \ref{table:sanity} showcases the 
accuracy of the numerical methods in this example.\footnote{For the LP method, 
the discretisation method from Appendix \ref{app:discrete} is used with $n = 18, 4, 2$ for $d=2, 3, 4$ respectively. For the NN method, the 
penalization $\theta$ is taken as the product measure of the marginals and $\gamma = 5000\cdot d$. 
For the implementation, feed-forward neural networks with 5 
layers and hidden dimension 32 are used, and for training we employed the Adam optimizer 
with $\beta_1 = 
0.99, \beta_2 = 0.995$ and batch size $2^{11},2^{13}, 2^{15}$ for $d=2, 3, 4$.}
Overall, especially for $d=2, 3$, the relative errors are quite small, below 
$1\%$ relative error (in absolute terms, depending on the random cost 
functions, the true optimal values were typically around $3$ to $10$).
For the LP method, higher error values are expected for increasing 
$d$, since fewer support points for the discretisation in each dimension 
can be used. The reason is that the total number of variables in the LP is 
limited due to working memory, and given by $\prod_{t=1}^2 \prod_{i=1}^d n_{t, 
i}$ where $n_{t, i}$ is the number of support points of the discretised 
approximation of $\mu_{t, i}$. A priori, the NN method does not have this 
drawback. The error for the NN method is governed by the term 
$\int \beta_\gamma^*(\frac{d\pi}{d\theta})\,d\theta$, where $\pi$ is an optimal 
coupling and $\theta$ the measure chosen for penalization. In particular when 
any optimizer $\pi$ is highly singular with respect to $\theta$, as in this 
example, this error term can increase sharply with increasing dimension as 
well. While theoretically, this can be overcome by increasing $\gamma$, the 
error of the numerical method has to be taken into account as well, see also 
Figure 
\ref{fg:convergence_analysis}.
\def\arraystretch{1.5}
\begin{table}
	\centering
	\caption{Relative error of the numerical methods}
	\begin{tabular}{ l | c | c | c  }
		Dimension $d$  & 2  & 3 & 4 \\\hline 
		LP & 0.04\% & 0.78\% & 3.12\%  \\
		NN & 0.08\% & 0.61\% & 2.05\% \\
	\end{tabular}\vspace{1mm}\\
	\caption*{Average relative error over 100 sample runs with varying cost 
		functions 
		is showcased for Example \ref{subsec:accuracy}.}
	\label{table:sanity}
\end{table}
\def\arraystretch{1}

\subsection{Real-world application: foreign exchange data} 
\label{sec:FX}
We close the examples section with an example using FX data. We 
work with option data on three currency pairs, $X_1 = \textrm{GBPUSD}$, $X_2 = 
\textrm{EURUSD}$, $X_3 = \textrm{EURGBP}$. The data was collected from a 
Bloomberg terminal on the 28 January 2019 for three tenors: 1y, 1.5y and 2 
years out. We converted\footnote{We are grateful to the Oxford-Man Institute 
for access to Bloomberg terminals and to Shen Wang for his help with the data 
conversion.} prices from FX specific convention to the strike convention used 
here and also converted put prices into call prices using the put-call parity, 
see Figure \ref{fg:Data}. Given the prevailing low interest rates and the illustrative nature of the example, we assumed all domestic and foreign interest rates are equal so no discounting was needed. 
We denote the prices of the three assets by $X_{t, i}$, where $t=0$ is the current exchange rate on 28/01/19 and $t=1,2,3$ corresponds to the three tenors above. To test our methodology, we study a synthetic spread 
process between two USD denominated exchange rates: $(X_{t, 1} - \frac{X_{0, 1}}{X_{0, 2}}X_{t, 2})_{t=1, 2, 3}$. In particular we 
calculate the range of arbitrage free prices for an Asian call option on this spread process over the time points $t=1, 2, 
3$, i.e., we consider the payoff 
\begin{equation}\label{eq:FXpayoff}
c(X) = \left(\frac{1}{3} \sum_{t = 1}^3 X_{t, 1} - \frac{X_{0, 1}}{X_{0, 2}}X_{t, 2}\right)^+. 
\end{equation}

\begin{figure}
	\caption{Foreign exchange option data on three different currency pairs. 
		Date of retrieval for the prices is January 28, 2019.}\vspace{4mm}
	
	\vspace{1mm}
	\begin{minipage}[b]{0.5\textwidth}
		\includegraphics[width=\textwidth,height=0.7\textwidth]{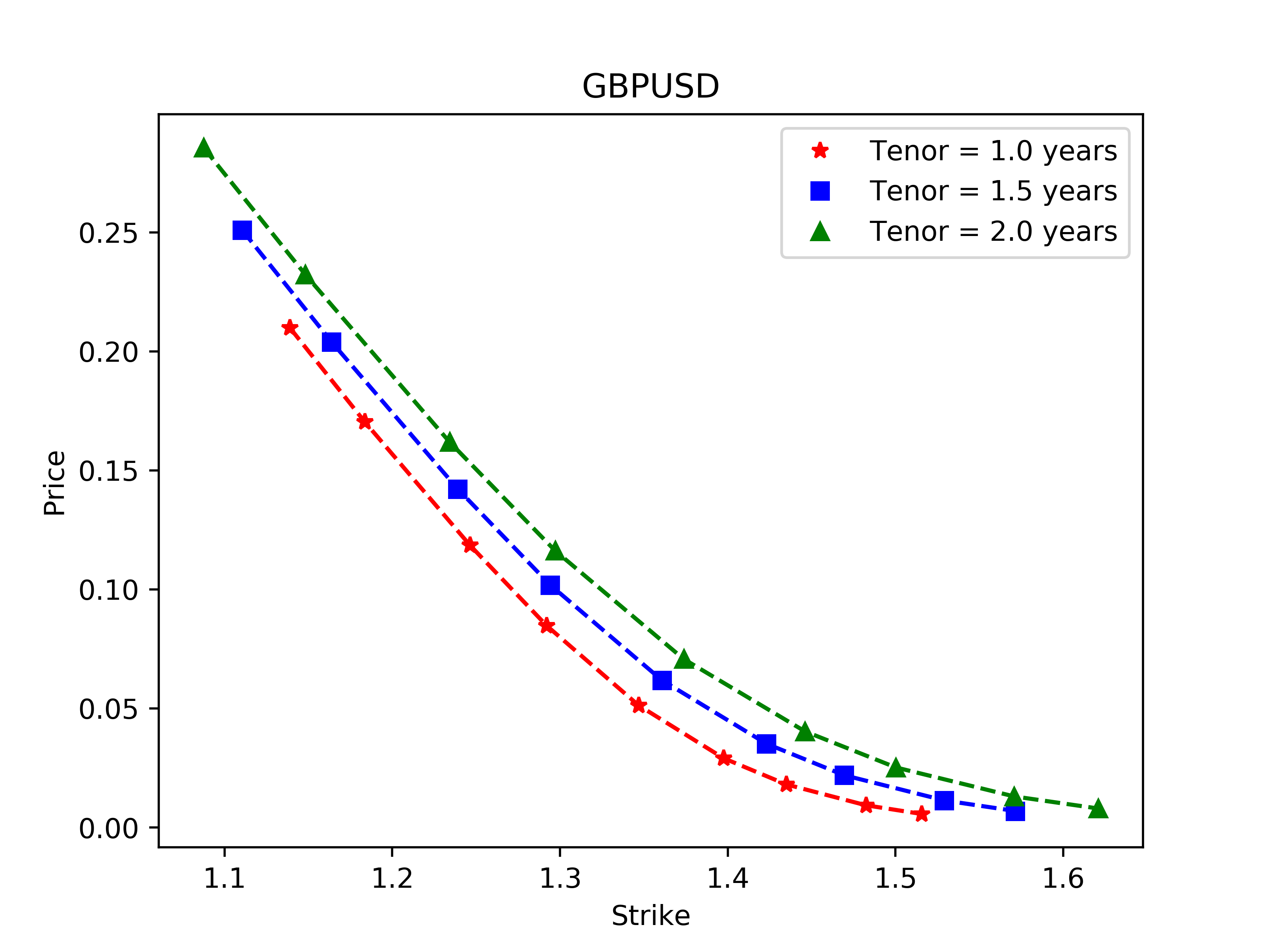}
	\end{minipage}\hspace{-5mm}
	\hfill
	\begin{minipage}[b]{0.5\textwidth} 
		\includegraphics[width=\textwidth,height=0.7\textwidth]{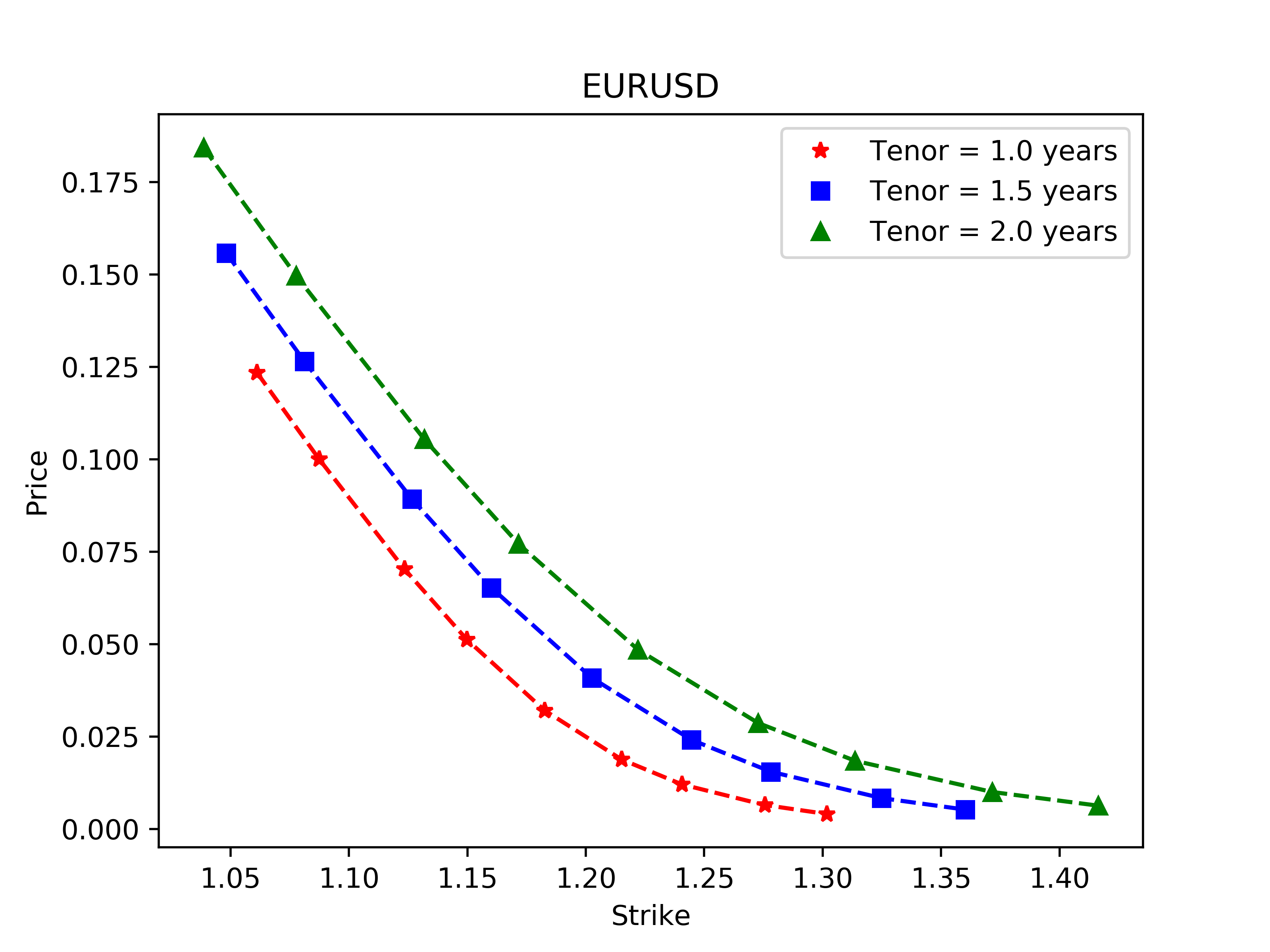}
	\end{minipage}\hspace{-5mm}
	\hfill
	
	\begin{minipage}[b]{0.5\textwidth}
		\includegraphics[width=\textwidth,height=0.7\textwidth]{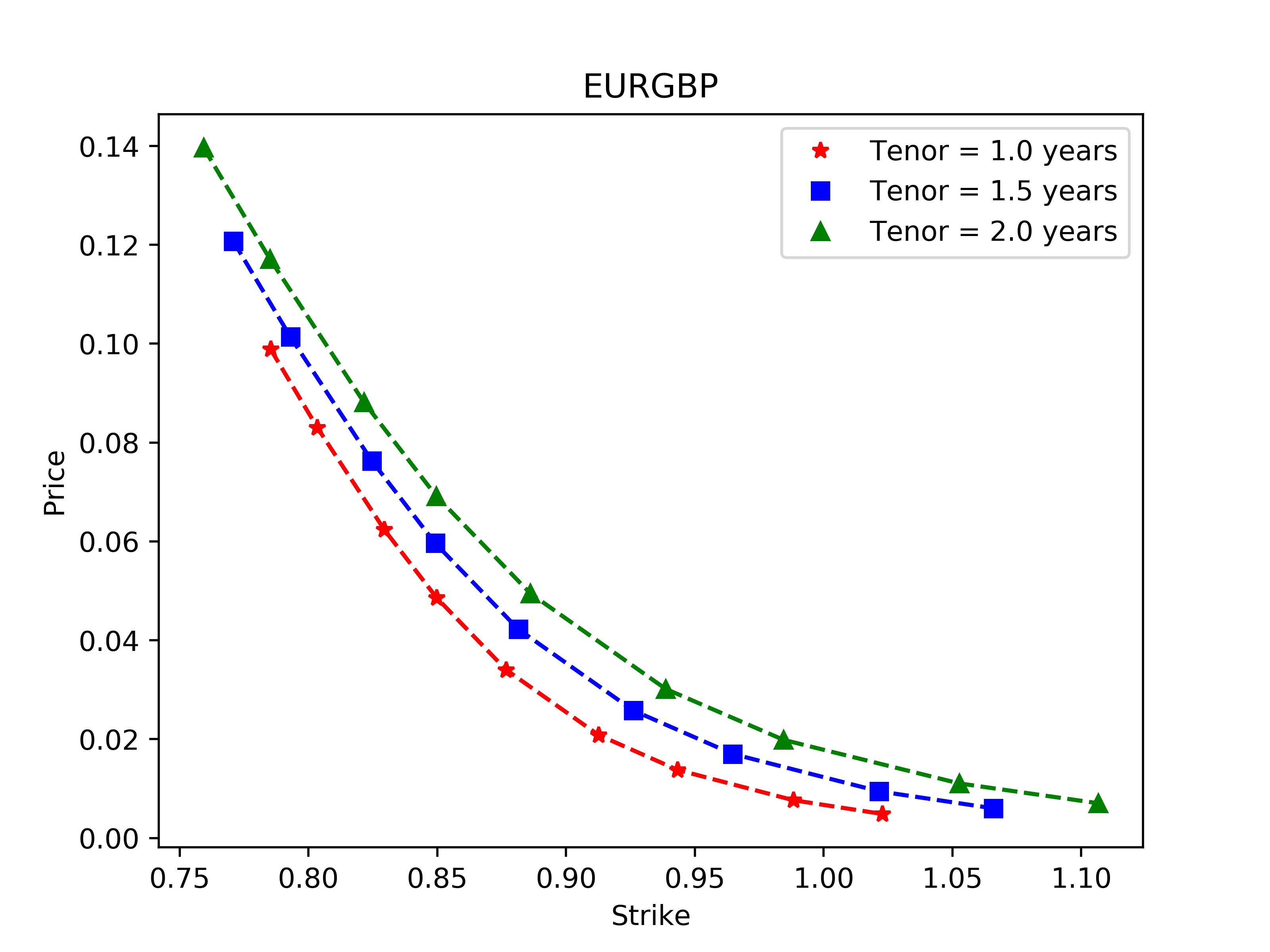}
	\end{minipage}\hspace{-5mm}
	\hfill
	\caption*{The table showcases data for call option prices for three 
	different currency pairs. Further, initial values at time point 0 are 
		given by $\textrm{GBPUSD} = 1.32$, $\textrm{EURUSD} = 1.14$, 
		$\textrm{EURGBP} = 1.15$.}
	\label{fg:Data}
\end{figure}
We employ the NN methodology to compute the optimal values. In practice, we 
modify slightly the formulation of $\Do^m$ in Section \ref{sec:NNtheory}. 
Instead of estimating the risk-neutral marginal distributions from the call/put 
prices and considering all static position $\varphi_{t, i} \in \mathfrak{N}_{l, 
1, m}$, we directly consider $\varphi_{t, i}$ which are linear combinations of 
traded call options. Figure \ref{fig:FX} displays the range of no-arbitrage 
prices under two different information structures: without and with the option 
prices for the third currency pair. We expect that the EURGBP prices, 
$X_3=X_{t, 2}/X_{t, 1}$, capture important information about the correlation 
structure between $X_1$ and $X_2$ which should be material for pricing of our 
Asian option even if $X_3$ is not explicitly present in its payoff. This is 
indeed true as seen from the price tightening in Figure \ref{fig:FX} between 
the upper and the lower bars. In particular, we see that the upper bound 
shrinks by nearly 50\% when the additional information is included. 
Furthermore, for both scenarios, we consider also the impact of the ability to hedge 
dynamically in the assets as compared to only taking static positions in the options. 
Again, this leads to a tightening of the bounds, albeit less pronounced. We 
believe that this example showcases the capacity of our methodology to capture, 
in a fully non-parametric but quantitative manner, the importance of market 
information for a given pricing problem. Naturally, its full potential should 
be explored on a much larger and more comprehensive range of market 
data/problems. This is left for future research.

\begin{figure}
	\caption{Price intervals for an Asian call option 
	written on the spread process 
	between $\rm{GBPUSD}$ and $\rm{EURUSD}$.}
	\includegraphics[width=\textwidth]{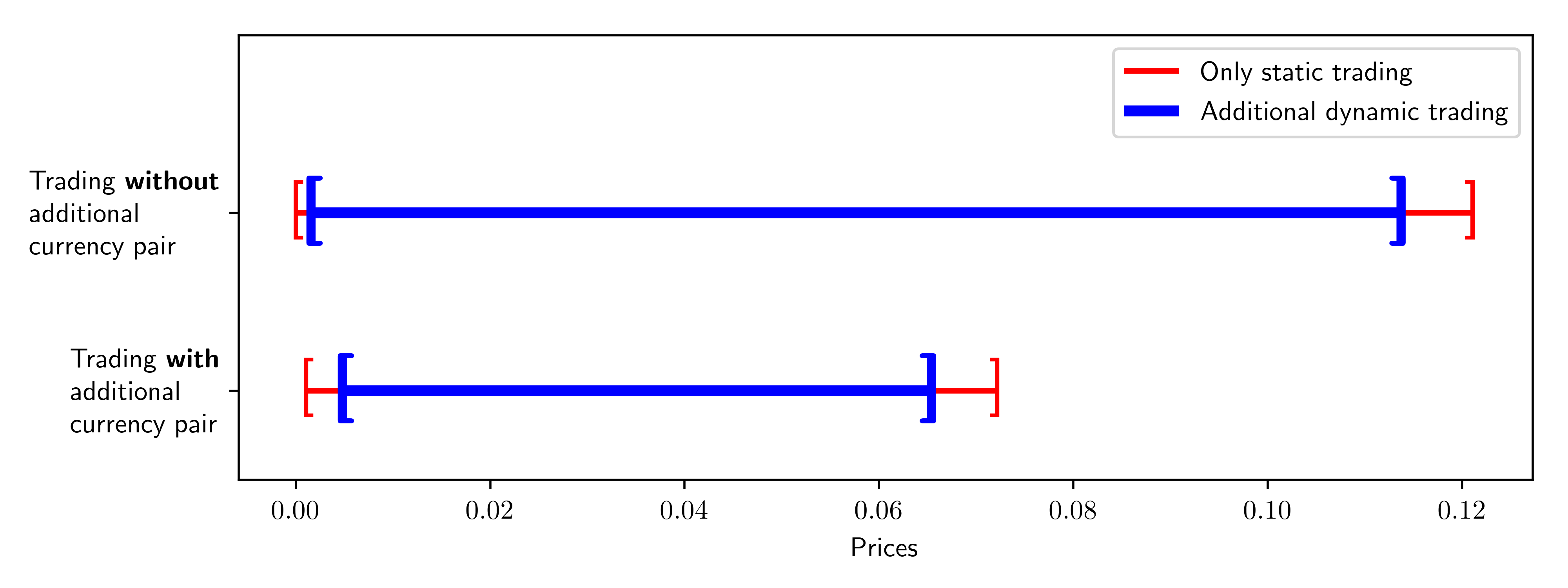}
	\caption*{Numerically computed ranges of arbitrage free prices for the Asian option with payoff in \eqref{eq:FXpayoff}. For the red 
	lines, all models (without martingale assumption), or equivalently without 
	dynamic trading are considered. For the blue lines, only martingale models 
	are considered. For the bottom lines, additional trading in call options 
	written on $X_{t, 3} = X_{t, 2}/X_{t, 1}$ is allowed.}
	\label{fig:FX}
\end{figure}

\section{A structural result on the covariance functional}
\label{sec:structure}
In this section we study a two-period model, i.e., $T=2$, and develop 
structural results for the optimizers. Our study was partly inspired by Figure 
\ref{fg:OptimizerSpread} where the two time step optimizer has the structure of 
a probability distribution on a line superimposed with the OT optimizer. We 
shall see in Theorem \ref{thm:structure} below that this structure is in fact 
universal, under certain assumptions on the marginal distributions. To make 
notation simpler, we write $X = (X_i)_{1\le i \le d}$, $Y= (Y_i)_{1\le i \le 
d}$ instead of $X_1 = (X_{1,i})$, $X_2= (X_{2,i})$, and $\mu=(\mu_i)$, 
$\nu=(\nu_i)$ instead of $\mv_1 = (\m_{1, i})$, $\mv_2 = (\m_{2, i})$. Hence we 
consider the one-step martingales $(X_i, Y_i)_{1 \le i \le d}$ with marginals 
$X_i \sim \mu_i$, $Y_i \sim \nu_i$. For each $\pi = \Lc(X,Y)\in \Mc(\mu,\nu)$,  
define $\pi^1 = \pi \circ X^{-1}$, $\pi^2 = \pi \circ Y^{-1}$ to be the 
$d$-dimensional marginals of $\pi$. We assume that all $\mu_i, \nu_i$ have 
finite second moments. Define $[d]=\{1,2,...,d\}$.

We will consider the maximization problem \eqref{def:mmot}  with the cost 
functional which concerns the mutual covariance of the value of assets at the 
terminal time
\begin{align}
c(Y)=\sum_{1\le i < j \le d} c_{ij}Y_iY_j, \quad\text{where} \quad c_{ij} \ge 0.
\end{align}
We can assume without loss of generality that every $Y_j$ is involved in 
$c(Y)$, that is, for each $j \in [d]$ there exists nonzero $c_{ij}$ or 
$c_{jk}$; otherwise we may simply ignore the $j$-th asset in our optimization 
problem.  We can regard $[d]$ as the set of nodes of a graph where $i,j$ is 
connected by an (undirected) edge if $c_{ij} > 0$. Then $[d]$ is decomposed 
into connected subgraphs, and it is clear that the MMOT problem can be 
decomposed accordingly. Therefore, without loss of generality we can assume 
that $[d]$ is connected.

For our structural result, we also introduce the following notion.
\begin{definition}[Linear Increment of Marginals (LIM)] We say that 
marginals 
$(\mu_i, \nu_i)_{1\le i \le d}$ satisfy {\rm LIM} if there exists a centered 
non-Dirac probability measure $\kappa$, and positive constants $a_1,...,a_d$ 
such that 
	$$ \nu_i = \mu_i * {a_i}_\# \kappa$$
	where ${a_i}_\# \kappa$ is the push-forward of $\kappa$ by the scaling map 
	$x \mapsto a_ix$. In other words, $\Lc(Y_i)= \Lc(X_i + a_iZ)$ where $Z \sim 
	\kappa$ is independent of $X$ and $\E[Z]=0$, $\mathbb{P}[Z\neq 0]>0$. 
\end{definition}
\begin{example} LIM holds when each pair of marginals $\mu_i, \nu_i$ are 
Gaussians 
with the same mean and increasing variance.
\end{example}

\begin{theorem}\label{thm:structure} Let $c(Y)= \sum_{1\le i < j \le d} 
c_{ij}Y_iY_j$ and assume  $c_{ij}$'s induce a connected graph on $[d]$. Suppose $(\mu_i, \nu_i)_{1\le i \le 
d}$ satisfy {\rm LIM} with constant $a = (a_1, ..., a_d)$. Let $L$ be the 
one-dimensional subspace of $\R^d$ spanned by $a$. Then every {\rm MMOT} 
 $\pi$ for the maximization problem \eqref{def:mmot}, if disintegrated as $\pi(dx,dy) = \pi_x(dy) 
\pi^1(dx)$, satisfies:
	\begin{enumerate}
		\item $\supp\, \pi_x \subset L+x \q \pi^1$ - almost every $x,$
		\item $\pi^1$ is an optimal transport plan in $\Pi(\mu_1,...,\mu_d)$ 
		for the maximization problem with the corresponding cost $c(X)= 
		\sum_{1\le i < j \le d}c_{ij} X_iX_j$.
	\end{enumerate}
	Moreover if $d=2$ or $3$ and the first marginals $(\mu_i)_i$ are continuous 
	(i.e., $\mu_i(\{x\})=0$ for all $x \in \R$ and $i \in [d]$), 
	then $\pi^1$ is 
	unique for every {\rm MMOT} $\pi$.
\end{theorem}
To prove the theorem, we shall need the following lemma.
\begin{lemma}\label{tangent}
	Let $c(x)= \sum_{1\le i < j \le d} c_{ij}x_ix_j$ and assume $c_{ij}$'s induce a connected graph on $[d]$. Let $\la_{ij}=\sqrt{c_{ij}\frac{a_j}{a_i}}$, 
	$\sigma_{ij}=\sqrt{c_{ij}\frac{a_i}{a_j}}$, and 
	$g_{ij}(x)=\frac{1}{2}\big{(}\la_{ij}x_i - \sigma_{ij}x_j\big{)}^2$ for 
	each $i<j$. Define $G(x)=\sum_{i<j} g_{ij}(x)$, and let $H_{x_0}(x) = 
	G(x_0)+ \nabla G(x_0) \cdot (x-x_0)$ be the affine tangent function of $G$ 
	at $x_0 \in \R^d$. Then $$\{x \in \R^d \,|\, G(x)=H_{x_0}(x) \} = x_0 + L,$$
	where $L$ is the one-dimensional subspace of $\R^d$ spanned by $a = 
	(a_1,...,a_d)$. 
\end{lemma}
\begin{proof} Note that $x \mapsto g_{ij}(x)$ is constant if $a_jx_i - a_ix_j$ 
is constant. Hence $G$ is constant on $x_0 + L$. Since $G$ is smooth and 
convex, this implies that $\nabla G$ is constant on $x_0 + L$, yielding $x_0 + 
L \subset K:=\{x \in \R^d \,|\, G(x)=H_{x_0}(x) \}$.
	
	Conversely, clearly $G$ is an affine function on $K$, and since $g_{ij}$ 
	are convex, all $g_{ij}$ are also affine on $K$. But any nonzero $g_{ij}$ 
	can be affine only when $a_jx_i - a_ix_j$ is constant. Since $x_0 \in K$ 
	and $[d]$ is connected, this implies that $K \subset x_0 + L$.
\end{proof}
\begin{proof}[Proof of Theorem \ref{thm:structure}] Let $x=(x_1,...,x_d)$, 
$y=(y_1,...,y_d) \in \R^d$. We will construct functions $\phi_i \in 
L^1(\mu_i)$, $\psi_j \in L^1(\nu_j)$, $h : \R^d \to \R^d$ such that
	\begin{align}\label{ineq}
	&\sum_{i=1}^d \phi_i(x_i) + \sum_{i=1}^d \psi_i(y_i) + h(x) \cdot (y-x) \ge 
	c(y) \quad on \quad \R^d \times \R^d,
	\end{align}
	but for any solution $\pi^* \in \mathcal{M}(\mu,\nu)$ to the problem 
	\eqref{def:mmot}, we have
	\begin{align}\label{eq}
	\sum_{i=1}^d \phi_i(x_i) + \sum_{i=1}^d \psi_i(y_i) + h(x) \cdot (y-x) = 
	c(y) \quad \pi^* - a.e. \,(x,y).
	\end{align}
	We shall call the triplet $(\phi_i, \psi_i, h)$ a dual optimizer, and 
	$\pi^*$ a multi-marginal martingale optimal transport (MMOT); see \cite{Lim}. Along 
	the proof, we will see that the equality \eqref{eq} implies that $y-x \in 
	L$.
	
	To begin, let $f_i \in L^1(\mu_i)$, $i=1,...,d$, be a dual optimizer for 
	the optimal transport with the cost $c(x)$, that is, for any optimal 
	transport $\gamma \in \Pi(\mu_1,...,\mu_d)$
	\begin{align}
	\label{otineq}&\sum_{i=1}^d f_i(x_i) \ge c(x) \quad \forall x \in \R^d,\\
	\label{oteq}&\sum_{i=1}^d f_i(x_i) = c(x) \quad \gamma - a.e. \, x.
	\end{align}
	For the existence of such a dual optimizer, see
	\cite{Vi03, Vi09}. Recall the functions $g_{ij}$ and $G$ in Lemma 
	\ref{tangent}, and note that $G(x)=-c(x) + \sum_{i=1}^d b_i x_i^2$ for some 
	$b_i \ge 0$. Define $\phi_i(x_i) = f_i(x_i)- b_ix_i^2$ and $\psi_i(y_i)=b_i 
	y_i^2$. Then the above may be rewritten as
	\begin{align}
	\label{otineq1}&-\sum_{i=1}^d \phi_i(x_i) \le G(x) \quad \forall x \in 
	\R^d,\\
	\label{oteq1}&-\sum_{i=1}^d \phi_i(x_i) = G(x) \quad \gamma - a.e. \, x.
	\end{align}
	Next, define $h(x) = - \nabla G(x)$, so that we have 
	\begin{align}\label{eq2}
	G(x)- h(x) \cdot (y-x) \le G(y), \text{ and the equality holds iff } y-x 
	\in L
	\end{align}
	by Lemma \ref{tangent}. With \eqref{otineq1} this implies \eqref{ineq}. Moreover, 
	notice that if $(x,y)$ satisfies the equality \eqref{eq}, then it holds $y-x \in L$ and the equality 
	\eqref{oteq}.
	
	Now we will construct a multi-marginal martingale transport $\pi^* \in \Mc(\mu,\nu)$ 
	such that $\pi^*$ is concentrated on the equality set in \eqref{eq}, that 
	is $\pi^*(P)=1$ where
	\begin{align*}
	P:=\{(x,y) \in \R^d \times \R^d \,|\, \sum_{i=1}^d \phi_i(x_i) + 
	\sum_{i=1}^d \psi_i(y_i) + h(x) \cdot (y-x) = c(y) \}.
	\end{align*}
	We also define $P_1:=\{x \in \R^d \,|\,\sum_{i=1}^d f_i(x_i) = c(x)\}$. In 
	order to construct $\pi^*(dx,dy) = \pi^*_x(dy) \pi^{*1}(dx)$, firstly set 
	$\pi^{*1}$ to be an optimal transport, i.e. $\pi^{*1} \in 
	\Pi(\mu_1,...,\mu_d)$ and $\pi^{*1}(P_1)=1$. Next, let $\sigma$ be the distribution of the vector $(a_1Z,\ldots, a_dZ)$ with $Z\sim\kappa$, and note that $\sigma(L)=1$ and $\sigma \in \Pi({a_1}_\# \kappa,...,{a_d}_\# \kappa)$.
		
	For each $x \in \R^d$, define the 
	kernel $\pi^*_x$ to be the $\sigma$ translated by $x$. As $\sigma$ has its 
	barycenter at $0$, $\pi^*_x$ is clearly a martingale kernel. Now to ensure 
	that $\pi^* \in \Mc(\mu,\nu)$, it remains to show that $\pi^{*2} \in 
	\Pi(\nu_1,...,\nu_d)$.  But notice that this follows from the facts 
	$\pi^{*1} \in \Pi(\mu_1,...,\mu_d)$, $\sigma \in \Pi({a_1}_\# 
	\kappa,...,{a_d}_\# \kappa)$, the definition of $\pi^*_x$, and finally the 
	assumption LIM, i.e. $\nu_i = \mu_i * {a_i}_\# \kappa$.

	Now observe that $\pi^*_x(L+x) =1$ and $\pi^{*1}(P_1)=1$ imply, by 
	\eqref{otineq1}, \eqref{oteq1}, and \eqref{eq2}, that $\pi^*(P)=1$. This 
	immediately implies the optimality of $\pi^*$ to the MMOT problem 
	\eqref{def:mmot} by the following standard argument: let $\pi \in 
	\Mc(\mu,\nu)$ be any multi-marginal martingale transport. By integrating both sides 
	of \eqref{ineq} by $\pi$, we get
	$$\sum_{i=1}^d \int \phi_i\, d\mu_i + \sum_{i=1}^d \int\psi_i \,d\nu_i \ge \int 
	c\, d\pi$$
	since $\int h(x) \cdot (y-x) \,\pi(dx,dy)=0$. On the other hand, as 
	$\pi^*(P)=1$ we get
	$$\sum_{i=1}^d \int \phi_i\, d\mu_i + \sum_{i=1}^d \int\psi_i \,d\nu_i = \int 
	c\, d\pi^*.$$
	Hence $\int c\, d\pi^* \ge \int c\, d\pi$, and the optimality of $\pi^*$ 
	follows. The argument also shows conversely that any solution $\pi^*$ must 
	be concentrated on $P$, and this implies $\pi^*_x(L+x) =1$ and 
	$\pi^{*1}(P_1)=1$ by \eqref{otineq1}, \eqref{oteq1}, \eqref{eq2}. But 
	$\pi^{*1}(P_1)=1$ precisely means that $\pi^{*1}$ is an optimal transport 
	as claimed in the second part of the theorem. 

	Lastly, we prove the uniqueness statement. Let $\pi$ be an MMOT and let 
	$\gamma=\pi \circ X^{-1}$. As we have just shown, $\gamma$ satisfies 
	\eqref{otineq}, \eqref{oteq} for some $f_i \in L^1(\mu_i)$, $i=1,...,d$. If 
	$d=2$, it is well known in optimal transport theory (see \cite{Vi03}) that 
	the contact set $P_1=\{x \in \R^2 \ | \ \sum_{i=1}^2 f_i(x_i) = c(x)\}$ is 
	a subset of a nondecreasing graph, that is 
	\[ (x,y), (x',y') \in P_1 \text{ and } x < x' \implies y \le y',\]
	and this property immediately implies that there exists a unique 
	probability measure concentrated on $P_1$ which respects the marginal 
	constraints $\mu_1, \mu_2$. This proves the uniqueness assertion for $d=2$.

	Now let $d=3$ and $P_1=\{x \in \R^3 \ | \ \sum_{i=1}^3 f_i(x_i) = c(x)\}$. By 
	permuting the indices $1,2,3$ if necessary, by connectedness there are two cases of cost 
	function
	\[ c(x)=c_{12}x_1x_2 +c_{13}x_1x_3+c_{23}x_2x_3, \quad \text{or} \quad 
	c(x)=c_{12}x_1x_2 +c_{23}x_2x_3, \]
	where $c_{ij} >0$. Again consider \eqref{otineq}, \eqref{oteq}. By the 
	standard technique, called Legendre-Fenchel transform, we can assume that 
	$f_i$'s are convex functions, and hence in particular $f_i$ is 
	differentiable $\mu_i$-a.s.. Let $A_i$ be the set of differentiable points 
	of $f_i$, $i=1,2,3$. Now assume $(x_1,x_2,x_3) \in P_1$ and $x_1 \in A_1$. 
	Then by the first-order condition, \eqref{otineq}, \eqref{oteq} implies 
	\begin{align}
	f_1'(x_1)= c_{12}x_2 + c_{13}x_3,\nonumber
	\end{align}
	where in the latter cost function case $c_{13}=0$.
	Let $Q_1(x_1):=\{(x_2,x_3) \in A_2 \times A_3 \ | \ f_1'(x_1)= c_{12}x_2 + 
	c_{13}x_3\}$, which is a linearly decreasing, or vertical, graph in 
	$x_2x_3$-plane. On the other hand, the following `conditional contact set'
	\[P_1(x_1) := \{(x_2,x_3) \in A_2 \times A_3 \ | \ \sum_{i=1}^3 f_i(x_i) = 
	c(x)\}
	\]
	is a nondecreasing graph as before. But notice that in fact $P_1(x_1)$ is a 
	graph of a nondecreasing function defined on $A_2$, since again 
	\eqref{otineq}, \eqref{oteq} implies
	\[f_2'(x_2)=c_{12}x_1 +c_{23}x_3.\]
	We conclude that the intersection
	\[P_1(x_1) \cap Q_1(x_1)\]
	consists of at most one element for $\mu_1$-almost every $x_1$, and this 
	implies that there exist two functions $x_2=\phi(x_1), x_3=\psi(x_1)$, 
	well-defined $\mu_1$-a.s., such that any probability measure concentrated 
	on $P_1$ is in fact concentrated on the set
	\[G:=\{(x_1,x_2,x_3) \ | \ x_2=\phi(x_1), x_3=\psi(x_1)\}.\]
	By standard averaging argument, this implies the uniqueness of $\gamma$.  
	This completes the proof of Theorem \ref{thm:structure}.
\end{proof}

\appendix
\section{Discretization}\label{app:discrete}
This section shows a sample discretization and formulation of an MMOT problem 
as an LP. We take the case $T=2$ for the spread option from Table 
\ref{table:marginals}. Recall that $\m_{1, 1}=\m_{1,2}=\Uc\big([-1,1]\big)$, 
$\m_{2, 1}=\Uc\big([-3,3]\big)$ and $\m_{2, 2}=\Uc\big([-2,2]\big)$. Define 
discrete approximating probability measures via \eqref{eq:DSdiscretisation}, 
i.e.,
\begin{equation}
\begin{split}\label{eq:dis_pcoc}
\alpha^i_k \q = \q\m_{1, i}^n\big(\big\{k/n\big\}\big)&\q :=\q 
\int_{[(k-1)/n,(k+1)/n)}\big(1-|nk-x|\big)\mu_{1,i}(dx),  \\
\beta^i_k \q = \q \m_{2, i}^n\big(\big\{k/n\big\}\big)&\q :=\q 
\int_{[(k-1)/n,(k+1)/n)}\big(1-|nk-y|\big)\mu_{2,i}(dy),
\end{split}
\end{equation}
so that $\m^n_{1, i}\leqcvx \m^n_{2, i}$ and $\Wc_1(\m^n_{t, i},\m_{t, i})\le 
1/n$ for $t,i=1, 2$.
Furthermore
\b*
&&\alpha^1_{-n}~=~\alpha^1_n~=~1/4n\quad \mbox{and}  \quad\alpha^1_i~=~1/2n 
\mbox{ for } -n+1\le i\le n-1, \\
&&\alpha^2_{-n}~=~\alpha^2_n~=~1/4n \quad \mbox{and}  \quad \alpha^2_j~=~1/2n 
\mbox{ for } -n+1\le j\le n-1, \\
&&\beta^1_{-3n}~=~\beta^1_{3n}~=~1/12n \quad \mbox{and}  \quad \beta^1_k~=~1/6n 
\mbox{ for } -3n+1\le k\le 3n-1, \\
&&\beta^2_{-2n}~=~\beta^2_{2n}~=~1/8n \quad \mbox{and}  \quad \beta^2_l~=~1/4n 
\mbox{ for } -2n+1\le l\le 2n-1.
\e*
Then $\po_{2/n}(\mv_1^n, \mv_2^n)$ in the mentioned case of the spread option 
as objective is given by the following \emph{linear program} (LP): 
\begin{align*}
&\max_{(p_{i,j,k,l}) \in\R_+^M} ~ \sum_{i=-n}^{n} \sum_{j=-n}^{n} 
\sum_{k=-3n}^{3n} \sum_{l=-2n}^{2n}  p_{i,j,k,l} |y^1_k-y^2_l|^p \\
\mbox{s.t. } &\sum_{j=-n}^{n} \sum_{k=-3n}^{3n} \sum_{l=-2n}^{2n} p_{i,j, k,l} 
~=~\alpha^1_{i}, \quad \mbox{for } i=-n,\ldots, n,\\
&\sum_{i=-n}^{n} \sum_{k=-3n}^{3n} \sum_{l=-2n}^{2n}  p_{i,j, k,l} 
~=~\alpha^2_{j}, \quad \mbox{for } j=-n,\ldots, n,\\
&\sum_{i=-n}^{n} \sum_{j=-n}^{n} \sum_{l=-2n}^{2n} p_{i,j, k,l} ~=~\beta^1_{k}, 
\quad \mbox{for } k=-3n,\ldots, 3n,\\
&\sum_{i=-n}^{n} \sum_{j=-n}^{n} \sum_{k=-3n}^{3n} p_{i,j, k,l} ~=~\beta^2_{l}, 
\quad \mbox{for } l=-2n,\ldots, 2n,\\
&\left| \sum_{k=-3n}^{3n} \sum_{l=-2n}^{2n} \big(y^1_k-x^1_i\big) 
p_{i,j,k,l}\right| ~\le~ 2/n \sum_{k=-3n}^{3n} \sum_{l=-2n}^{2n} p_{i,j,k,l}, \\
&\hspace{8cm}\quad \mbox{for } i, j =-n,\ldots, n, \\
& \left|\sum_{k=-3n}^{3n} \sum_{l=-2n}^{2n} \big(y^2_l-x^2_j\big) 
p_{i,j,k,l}\right| ~\le~ 2/n \sum_{k=-3n}^{3n} \sum_{l=-2n}^{2n} p_{i,j,k,l}, \\
&\hspace{8cm}\quad \mbox{for } i, j=-n,\ldots, n,
\end{align*}
where
\b*
&&x^1_i~=~i/n, \quad \mbox{for } i = -n,\ldots, n, \\
&&x^2_j~=~ j/n, \quad \mbox{for } j = -n,\ldots, n, \\
&&y^1_k~=~k/n, \quad \mbox{for } k= -3n,\ldots, 3n, \\
&&y^2_l~=~l/n, \quad \mbox{for } l = -2n,\ldots, 2n,
\e* 
and $M:=(2n+1)^2(6n+1)(4n+1)$,  $N:=(6n+1)(4n+1)$. 

%
% ---- Bibliography ----
%
\bibliographystyle{apalike}

\bibliography{bib0}

\begin{thebibliography}{}

\bibitem[Alfonsi et~al., 2019]{ACJ:19}
Alfonsi, A., Corbetta, J., and Jourdain, B. (2019).
\newblock Sampling of probability measures in the convex order and and
  computation of robust option price bounds.
\newblock {\em Int. J. Theo. App. Finance}, 22(03):1950002.

\bibitem[Backhoff-Veraguas and Pammer, 2019]{BVP:19}
Backhoff-Veraguas, H. and Pammer, G. (2019).
\newblock Stability of martingale optimal transport and weak optimal transport.
\newblock arXiv:1904.04171.

\bibitem[Baker, 2012]{Baker:12}
Baker, D. (2012).
\newblock {\em Martingales with specified marginal}.
\newblock PhD thesis, Universit\'e Pierre et Marie-Curie - Paris VI.

\bibitem[Bartl et~al., 2017]{bartl2017duality}
Bartl, D., Cheridito, P., Kupper, M., and Tangpi, L. (2017).
\newblock Duality for increasing convex functionals with countably many
  marginal constraints.
\newblock {\em Banach J. Math. Analysis}, 11(1):72--89.

\bibitem[Beiglb{\"o}ck et~al., 2017a]{BCH:17}
Beiglb{\"o}ck, M., Cox, A., and Huesmann, M. (2017a).
\newblock {O}ptimal {T}ransport and {S}korokhod {E}mbedding.
\newblock {\em Invent. Math.}, 208(2):327--400.

\bibitem[Beiglb\"ock et~al., 2013]{BHLP}
Beiglb\"ock, M., Henry-Labord\'ere, P., and Penkner, F. (2013).
\newblock Model-independent bounds for option prices--a mass transport
  approach.
\newblock {\em Financ. Stoch.}, 17(3):477--501.

\bibitem[Beiglb{\"o}ck et~al., 2017b]{BNT:17}
Beiglb{\"o}ck, M., Nutz, M., and Touzi, N. (2017b).
\newblock Complete duality for martingale optimal transport on the line.
\newblock {\em Ann. Probab.}, 45(5):3038--3074.

\bibitem[Benamou et~al., 2015]{benamou2015iterative}
Benamou, J.-D., Carlier, G., Cuturi, M., Nenna, L., and Peyr{\'e}, G. (2015).
\newblock Iterative bregman projections for regularized transportation
  problems.
\newblock {\em SIAM Journal on Scientific Computing}, 37(2):A1111--A1138.

\bibitem[Breeden and Litzenberger, 1978]{BreedenLitzenberger:78}
Breeden, D.~T. and Litzenberger, R.~H. (1978).
\newblock Prices of state-contingent claims implicit in option prices.
\newblock {\em J. Business}, 51(4):621--651.

\bibitem[Brown et~al., 2001]{brown2001robust}
Brown, H., Hobson, D., and Rogers, C. (2001).
\newblock Robust hedging of barrier options.
\newblock {\em Math. Finance}, 11(3):285--314.

\bibitem[Burzoni et~al., 2019]{bfhmo}
Burzoni, M., Frittelli, M., Hou, Z., Maggis, M., and Ob\l{}\'{o}j, J. (2019).
\newblock Pointwise arbitrage pricing theory in discrete time.
\newblock {\em Math. Oper. Res.}, 43(3):1034--1057.

\bibitem[Chacon, 1977]{Chacon:77}
Chacon, R.~V. (1977).
\newblock Potential processes.
\newblock {\em Trans. Amer. Math. Soc.}, 226:39--58.

\bibitem[Cox and Ob{\l}{\'o}j, 2011]{cox2011robusta}
Cox, A. and Ob{\l}{\'o}j, J. (2011).
\newblock Robust pricing and hedging of double no-touch options.
\newblock {\em Finance Stoch.}, 15(3):573--605.

\bibitem[Cuturi, 2013]{cuturi2013sinkhorn}
Cuturi, M. (2013).
\newblock Sinkhorn distances: Lightspeed computation of optimal transport.
\newblock In {\em Advances in neural information processing systems}, pages
  2292--2300.

\bibitem[De~March, 2018]{de2018entropic}
De~March, H. (2018).
\newblock Entropic approximation for multi-dimensional martingale optimal
  transport.
\newblock arXiv:1812.11104.

\bibitem[Dolinsky and Soner, 2014]{DolinskySoner:2014}
Dolinsky, Y. and Soner, H. (2014).
\newblock Martingale optimal transport and robust hedging in continuous time.
\newblock {\em Probab. Theory Relat. Fields}, 160(1-2):391--427.

\bibitem[Eckstein and Kupper, 2019]{Eckstein2019}
Eckstein, S. and Kupper, M. (2019).
\newblock Computation of optimal transport and related hedging problems via
  penalization and neural networks.
\newblock {\em Appl. Math. Opt.}
\newblock (online) DOI: 10.1007/s00245-019-09558-1.

\bibitem[Fournier and Guillin, 2015]{fournier2015rate}
Fournier, N. and Guillin, A. (2015).
\newblock On the rate of convergence in wasserstein distance of the empirical
  measure.
\newblock {\em Probab. Theory Relat. Fields}, 162(3-4):707--738.

\bibitem[Galichon et~al., 2014]{galichon2014stochastic}
Galichon, A., Henry-Labord\'ere, P., and Touzi, N. (2014).
\newblock A stochastic control approach to no-arbitrage bounds given marginals,
  with an application to lookback options.
\newblock {\em Ann. Appl. Probab.}, 24(1):312--336.

\bibitem[Ghoussoub et~al., 2019]{ghoussoub2019structure}
Ghoussoub, N., Kim, Y.-H., and Lim, T. (2019).
\newblock Structure of optimal martingale transport plans in general
  dimensions.
\newblock {\em Ann. Probab.}, 47(1):109--164.

\bibitem[Gulrajani et~al., 2017]{gulrajani2017improved}
Gulrajani, I., Ahmed, F., Arjovsky, M., Dumoulin, V., and Courville, A.~C.
  (2017).
\newblock Improved training of wasserstein gans.
\newblock In {\em Advances in neural information processing systems}, pages
  5767--5777.

\bibitem[Guo and Ob\l\'oj, 2019]{GO}
Guo, G. and Ob\l\'oj, J. (2019).
\newblock Computational methods for martingale optimal transport problems.
\newblock {\em Ann. Appl. Probab.}, 29(9):3311--3347.

\bibitem[{Henry-Labord{\`e}re}, 2013]{phlautomated}
{Henry-Labord{\`e}re}, P. (2013).
\newblock Automated option pricing: Numerical methods.
\newblock {\em Int. J. Theo. App. Finance}, 16(8):1350042.

\bibitem[Hobson, 1998]{hobson1998robust}
Hobson, D. (1998).
\newblock Robust hedging of the lookback option.
\newblock {\em Finance Stoch.}, 2(4):329--347.

\bibitem[Hou and Ob\l{}\'{o}j, 2018]{hou2015robust}
Hou, Z. and Ob\l{}\'{o}j, J. (2018).
\newblock Robust pricing-hedging dualities in continuous time.
\newblock {\em Finance Stoch.}, 22:511--567.

\bibitem[Knight, 1921]{Knight:21}
Knight, F. (1921).
\newblock {\em Risk, Uncertainty and Profit}.
\newblock Boston: Houghton Mifflin.

\bibitem[Kramkov and Xu, 2019]{KramkovXu:19}
Kramkov, D. and Xu, Y. (2019).
\newblock An optimal transport problem with backward martingale constraints
  motivated by insider trading.
\newblock arXiv:1906.03309.

\bibitem[Lim, 2016]{Lim}
Lim, T. (2016).
\newblock Multi-martingale optimal transport.
\newblock arXiv:1611.01496.

\bibitem[L{\"u}tkebohmert and Sester, 2018]{lutkebohmert2018tightening}
L{\"u}tkebohmert, E. and Sester, J. (2018).
\newblock Tightening robust price bounds for exotic derivatives.
\newblock {\em Available at SSRN 3290503}.

\bibitem[Ob{\l}{\'o}j, 2004]{obloj2004skorokhod}
Ob{\l}{\'o}j, J. (2004).
\newblock The {S}korokhod embedding problem and its offspring.
\newblock {\em Probability Surveys}, 1:321--392.

\bibitem[Seguy et~al., 2018]{seguy2018large}
Seguy, V., Damodaran, B.~B., Flamary, R., Courty, N., Rolet, A., and Blondel,
  M. (2018).
\newblock Large-scale optimal transport and mapping estimation.
\newblock In {\em ICLR 2018-International Conference on Learning
  Representations}, pages 1--15.

\bibitem[Strassen, 1965]{strassen1965existence}
Strassen, V. (1965).
\newblock The existence of probability measures with given marginals.
\newblock {\em Ann. Math. Stat.}, 36(2):423--439.

\bibitem[Villani, 2003]{Vi03}
Villani, C. (2003).
\newblock {\em Topics in optimal transportation}, volume~58 of {\em Graduate
  Studies in Mathematics}.
\newblock American Mathematical Society, Providence, RI.

\bibitem[Villani, 2009]{Vi09}
Villani, C. (2009).
\newblock {\em Optimal Transport. Old and New}, volume 338 of {\em Grundlehren
  der mathematischen Wissenschaften}.
\newblock Springer.

\bibitem[Wang et~al., 2013]{Wang2013}
Wang, R., Peng, L., and Yang, J. (2013).
\newblock Bounds for the sum of dependent risks and worst value-at-risk with
  monotone marginal densities.
\newblock {\em Finance Stoch.}, 17(2):395--417.

\bibitem[Wiesel, 2019]{Wiesel:19}
Wiesel, J. (2019).
\newblock Continuity of the martingale optimal transport problem on the real
  line.
\newblock arXiv:1905.04574.

\bibitem[Zaev, 2015]{Zaev}
Zaev, D.~A. (2015).
\newblock On the monge-kantorovich problem with additional linear constraints.
\newblock {\em Math. Notes}, 98(5):725--741.

\end{thebibliography}

\end{document}